\documentclass[oneside,english]{amsart}
\usepackage[T1]{fontenc}
\usepackage[latin9]{inputenc}
\usepackage{geometry}
\usepackage{float}
\usepackage{textcomp}
\usepackage{multirow}
\usepackage{amsthm}
\usepackage{amstext}
\usepackage{amssymb}
\usepackage{esint}
\usepackage{amsmath}

\makeatletter

\providecommand{\tabularnewline}{\\}

\numberwithin{equation}{section}
\numberwithin{figure}{section}
\theoremstyle{plain}

\newtheorem{thm}{Theorem}
  \theoremstyle{plain}
  \newtheorem{cor}{Corollary}
  \theoremstyle{plain}
  \newtheorem{lem}{Lemma}
  \theoremstyle{remark}
  
  \theoremstyle{plain}
  \newtheorem{prop}{Proposition}
  \theoremstyle{plain}
  \newtheorem{defn}{Definition}

\makeatother

\usepackage{babel}

\begin{document}

\title{Subconvexity for a double Dirichlet series and non-vanishing of $L$-functions}

\author{Alexander Dahl}


\address{
Department of Mathematics and Statistics, York University, 4700 Keele Street, Toronto, Ontario, Canada M3J 1P3
}

\email{aodahl@yorku.ca}

\begin{abstract}
We study a double Dirichlet series of the form $\sum_{d}L(s,\chi_{d}\chi)\chi'(d)d^{-w}$,
where $\chi$ and $\chi'$ are quadratic Dirichlet characters with
prime conductors $N$ and $M$ respectively. A functional equation
group isomorphic to the dihedral group of order 6 continues the function
meromorphically to $\mathbb{C}^{2}$. A convexity bound at the central
point is established to be $(MN)^{3/8+\varepsilon}$ and a subconvexity
bound of $(MN(M+N))^{1/6+\varepsilon}$ is proven. The developed theory is used
to prove an upper bound for the smallest positive integer $d$ such
that $L(1/2,\chi_{dN})$ does not vanish, and further applications of subconvexity bounds to this problem are presented.
\end{abstract}

\maketitle

\section{Introduction}

The use of multiple Dirichlet series in number theoretic problems,
as well as the intrinsic structure they possess, have evolved the
subject to be a study in its own right. Indeed, Weyl group multiple
Dirichlet series are considered fundamental objects whose structures
are intimately linked to their analytic features \cite{weylgroup}.
In this paper we are interested in the size of a certain type of double
Dirichlet series outside the region of absolute convergence, and some
of its arithmetic consequences. For ordinary $L$-functions there
is the well-known concept of convexity bound as a generic upper bound
in the critical strip, and improvements in the direction of the Lindelöf
hypothesis often have deep implications. In the situation of multiple
Dirichlet series, the notion of convexity is not so obvious to formalize.
Nevertheless there is a fairly natural candidate for a ``trivial''
upper bound of a multiple Dirichlet series, and it is an important
problem to improve this estimate. This has first been carried out in \cite{blomer} for the Dirichlet series defined by
\[
Z(s,w)=\zeta^{(2)}(2s+2w-1)\sum_{d\text{ odd}}L^{(2)}(s,\chi_{d})d^{-w}
\]
at $s=1/2+it$, $w=1/2+iu$ simultaneously in the $t,u$-aspect, where
the $^{(2)}$ superscript denotes that we are removing the Euler factor
at $2$, and $\chi_{d}$ is the Jacobi symbol $(\frac{d}{\cdot})$.
Here we study a non-archimedean analogue. We can expand the $L$-function
in the summand as a sum indexed by $n$. Instead of twisting by $n^{it}$,
$d^{iu}$, we twist by quadratic primitive Dirichlet characters $\chi(n)$,
$\chi'(d)$ of conductors $N,M$, respectively. More precisely, we
consider
\[
Z(s,w;\chi,\chi'):=\sum_{\substack{d\geq1\\
(d,2MN)=1
}
}\frac{L^{(2MN)}(s,\chi_{d_{0}}\chi)\chi'(d)P_{d_{0},d_{1}}^{(\chi)}(s)}{d^{w}}
\]
for sufficiently large $\Re s,\Re w$, where we write $d=d_{0}d_{1}^{2}$
with $d_{0}$ squarefree. The $P$ factors are a technical complication
necessary in order to construct functional equations, which we accomplish
thus: if we expand the $L$-function in the numerator and switch the
summation, we can relate it to a similar double Dirichlet series with
the arguments and the twisting characters interchanged, with modified
correction factors. Next, we have a functional equation obtained by
application of the functional equation for Dirichlet $L$-functions,
which maps $(s,w)$ to $(1-s,s+w-{\textstyle \frac{1}{2}})$. Applying
the switch of summation formula to this, we obtain a functional equation
mapping $(s,w)$ to $(s+w-{\textstyle \frac{1}{2}},1-w)$. These continue
$Z$ to the complex plane except for the polar lines $s=1$, $w=1$,
and $s+w=3/2$. They also generate a group isomorphic to $D_{6}$,
the dihedral group of order 6, which is isomorphic to the Weyl group
of the root system of type $A_{2}$. For this reason $Z(s,w;\chi,\chi')$
is considered a type $A_{2}$ multiple Dirichlet series \cite{weylgroup}.
Interestingly, such objects are Whittaker coefficients in the Fourier
expansion of a $GL(3)$ Eisenstein series of minimal parabolic type
on a metaplectic cover of an algebraic group whose root system is
the dual root system of $A_{2}$ \cite{weyl-iii}. In fact, it is
conjectured that a similar relationship is true between multiple Dirichlet
series of type $A_{r}$ and $GL(r+1)$ Eisenstein series. Going down
in dimension, type $A_{2}$ multiple Dirichlet series can be constructed
from 1/2-integral weight Eisenstein series \cite{eisenstein-half-integral}.
As for the correction polynomials $P_{d_{0},d_{1}}^{(\chi)}(s)$,
their existence, uniqueness, and construction were studied extensively
in \cite{bfh} and \cite{DGH03} in the more general setting of $GL(r)$
multiple Dirichlet series. They are unique in the case
of $r$ up to 3.

The notion of convexity is no longer canonical in the case of double
Dirichlet series, since our initial bounds for $Z(s,w;\chi,\chi')$
depend on what we know about bounds on their coefficients, and our
knowledge here is only partial. Nonetheless, if we use the Lindelöf
hypothesis on average (cf.\ Theorem \ref{thm:HB-fouth-moment}),
then through careful choice of initial bounds and functional equation
applications (cf.\ \S\ref{sub:convexity-bound}), we obtain the bound
\[
Z({\textstyle \frac{1}{2}},{\textstyle \frac{1}{2}};\chi,\chi')\ll(MN)^{3/8+\varepsilon},
\]
which we call the convexity bound. In this work, we present the following
subconvexity result.
\begin{thm}
\label{thm:subconvexity}For quadratic Dirichlet characters $\chi$
and $\chi'$ of conductors $N$ and $M$ which are prime or unity,
and for $\varepsilon>0$, we have the bound
\[
Z({\textstyle \frac{1}{2}},{\textstyle \frac{1}{2}}+it;\chi,\chi')\ll_{\varepsilon}(1+\vert t\vert)^{2/\varepsilon}(MN(M+N))^{1/6+\varepsilon}.
\]

\end{thm}
We first point out that the $t$-aspect bound can be drastically improved,
but our focus here is on bounds in the moduli aspect.
For purposes of comparison with the convexity bound, we point out
that, via the geometric-arithmetic mean inequality, we have $(MN)^{3/8+\varepsilon}\ll(MN(M+N))^{1/4+\varepsilon}$.
This result is comparable to the subconvexity bound obtained by V.
Blomer in the archimedean case, particularly $\vert sw(s+w)\vert^{1/6+\varepsilon}$
for $\Re s=\Re w={\textstyle \frac{1}{2}}$, an improvement upon the
convexity bound with 1/4 replaced by 1/6.

A useful arithmetic application of the theory developed for $Z(s,w;\chi,\chi')$
is finding a bound for the least $d$ such that $L({\textstyle \frac{1}{2}},\chi_{dN})$
does not vanish, where $N$ is a large fixed prime. It is expected
that all ordinates of any zeros of $L(s,\chi)$ on the critical line
are linearly independent over the rationals, so that in particular
it is expected that $L({\textstyle \frac{1}{2}},\chi)$ is nonzero
for any $\chi$. Random matrix theory provides further evidence to
this conjecture: The lowest zero in families of $L$-functions (such
as $L(s,\chi_{d})$) is expected to be distributed like the ``smallest''
eigenvalue (i.e., closest to unity) of a certain matrix family (depending
on the family of $L$-function). Since the corresponding measure vanishes
at zero, it suggests that the smallest zero of a Dirichlet $L$-function
is ``repelled'' from the real axis. In the case of quadratic twists
of $GL(2)$ automorphic forms, in particular twists of elliptic curves,
there is a connection to Waldspurger's theorem \cite{waldspurger}
which states that $L({\textstyle \frac{1}{2}},f\times\chi)$ is proportional
to the squares of certain Fourier coefficients of a half-integral
weight modular form, uniformly in $\chi$. We focus here on the simpler
case of twists of Dirichlet characters. In particular, we have the
following theorem which is proven using the theory of this double Dirichlet series.
\begin{thm}
\label{thm:nonvanishing}Let $N$ be an odd prime, and let $D(N)=d$
denote the smallest positive integer such that $L({\textstyle \frac{1}{2}},\chi_{dN})$
does not vanish. Then we have $D(N)\ll N^{1/2+\varepsilon}$.
\end{thm}
It is not entirely obvious what should be regarded as the trivial
bound in this situation. The most natural approach to non-vanishing
would be to prove a lower bound for the first moment $\sum_{d\asymp X}L({\textstyle \frac{1}{2}},\chi_{dN})$
for $X$ as small as possible in terms of $N$. A straightforward
argument produces a main term of size $X\log X$ and an error term
of size $O((NX)^{1/2+\varepsilon})$ which suggests the trivial bound
$N^{1+\varepsilon}$ for the first non-vanishing twist. This is perhaps
unexpectedly weak, since the same bound holds for degree 2 $L$-functions
$L({\textstyle \frac{1}{2}},f\times\chi_{d})$, where $f$ is an automorphic
form of level $N$. Nevertheless it is not completely obvious how
to improve this in the case of Dirichlet characters $\chi_{N}$. Here we follow a modified version of a method presented in \cite{hk}: Let $X$ be a large positive number, and $h(y)$ be a smooth
non-negative function with support on $[1,2]$. By Mellin inversion,
we have
\[
\int_{(2)}\tilde{h}(w)Z({\textstyle \frac{1}{2}},w;\chi_{N},\psi_{1})X^{w}dw\approx\sum_{\substack{d=1}
}^{\infty}L({\textstyle \frac{1}{2}},\chi_{dN})h(d/X),
\]
where $\tilde{h}$ denotes the Mellin transform of $h$, and $\psi_{1}$
denotes the trivial character. We move the contour to $\Re w=-\varepsilon$,
picking up a double pole at $w=1$. If we apply a symmetric functional equation (\ref{eq:functional-3}) and a bound resulting from use of the Lindel\"of principle on average (\ref{eq:convex-trivial-2}) to the resulting integral, then we have

\begin{equation}
\sum_{\substack{d=1}
}^{\infty}L({\textstyle \frac{1}{2}},\chi_{dN})h(d/X)\approx a_{N}X\log X+b_{N}X+O(N^{1/2+\varepsilon})\label{eq:nonvanishing-asymptotic-approx}
\end{equation}
for some coefficients $a_{N},b_{N}$. The idea now is to bound $a_{N}$
from below in terms of $N$, and choose $X$ so that the main term
is greater than the error term. Then it cannot be that $L({\textstyle \frac{1}{2}},\chi_{dN})$
vanishes for all $0<d<X$ on the left-hand side. We prove the bounds
$a_{N},b_{N}\gg N^{-\varepsilon}$ (cf. Theorem \ref{thm:L-moment-asymptotic}),
and thus we can choose $X=N^{1/2+\varepsilon}$. The power of the
asymptotic formula (\ref{eq:nonvanishing-asymptotic-approx}) is that
$a_{N},b_{N}$ can be bounded below by finding an asymptotic for $\sum_{d\asymp X}L({\textstyle \frac{1}{2}},\chi_{dN})$,
but now we can take $X$ to be very large.

Of course, in the previous discussion there are significant details
suppressed for brevity, which include the aforementioned correction
factors and the error term arising from truncation of the $d$-sum.
Most nontrivial, however, are the bounds for $a_{N},b_{N}$ given
by Theorem \ref{thm:L-moment-asymptotic}, requiring careful treatment.
The techniques used are similar to those in \cite{jutila},
in which an asymptotic formula for
$
\sum_{0<d\leq X}L^{k}({\textstyle \frac{1}{2}},\chi_{d})
$
is proved, where the sum is over fundamental discriminants.

In order to prove the subconvexity bound Theorem \ref{thm:subconvexity},
we follow techniques similar to those used in \cite{blomer}, but
we must deal with some new complications. As mentioned, we need special
correction factors in order to obtain functional equations, whereas
in the archimedean case, the correction factors are far simpler. These
correction factors appear due to the fact that the characters involved
will only be primitive when the summation variable is a fundamental
discriminant. Also, the introduction of character twists complicates
the functional equations considerably: a pair of twisting characters
$(\psi,\psi')$ will not be static under the variable transformations
of the functional equations.

We can iterate the functional equations to obtain one under the map
$(s,w)\mapsto(1-s,1-w)$. Using techniques similar to those in the
case of $L(s,\chi)$ (cf.\ \cite{IK04} Theorem 5.3), we obtain the
approximate functional equation
\[
Z({\textstyle \frac{1}{2}},{\textstyle \frac{1}{2}};\chi,\chi')\approx\sum_{d<X}\frac{L({\textstyle \frac{1}{2}},\chi_{d}\chi)\chi'(d)}{d^{1/2}}+\sum_{d\leq M^{2}N/X}\frac{L({\textstyle \frac{1}{2}},\chi_{d}\chi'\chi_{M})\chi\chi_{N}(d)}{d^{1/2}},
\]
(cf.\ Lemma \ref{lem:approx-functional}). We can further apply the
approximate functional equation for $L$-functions,
\[
L({\textstyle \frac{1}{2}},\chi_{d}\chi)\approx\sum_{n\leq(dq)^{1/2}}\frac{\chi(n)}{n^{1/2}}.
\]
We hence see that we can roughly express $Z({\textstyle \frac{1}{2}},{\textstyle \frac{1}{2}};\chi,\chi')$
as a sum of double finite sums of the form
\[
S(P,Q;\chi,\chi'):=\sum_{d\leq P}\sum_{n\leq Q}\frac{\chi_{d}(n)\chi(n)\chi'(d)}{d^{1/2}n^{1/2}}
\]
for various character pairs $(\chi,\chi')$. Finally, we apply Heath-Brown's
large sieve estimate (cf.\ Corollary \ref{cor:heath-brown-large-sieve})
to obtain the subconvexity result.

We note that it is possible to extend these results beyond the case
where $M$ and $N$ are prime to arbitrary numbers, but the coefficients
of the functional equations in Theorems \ref{thm:functional-1} and
\ref{thm:functional-2} would become considerably more complicated
due to the Euler factors involved. We thus only treat the case of
prime moduli here to simplify the presentation.

\medskip{}

\textbf{Notation. }The variable $\varepsilon$ will always denote
a sufficiently small positive number, not necessarily the same at
each occurrence, and the variable $A$ will denote a sufficiently
large positive number, not necessarily the same at each occurrence.
The numbers $M$ and $N$ will always denote natural numbers that
are either odd primes or unity, possibly equal. For a real function
$f$, we denote its Mellin transform by $\tilde{f}$. The trivial
character modulo unity will be denoted by $\psi_{1}$, and the primitive
character of conductor 4 shall be denoted $\psi_{-1}$. As for the
primitive characters modulo 8, we define $\psi_{2}$ as the character
that is unity at exactly 1 and 7, and we set $\psi_{-2}=\psi_{2}\psi_{-1}$.
If $\chi$ is a character, we use the notation $C_{\chi}$ to denote
its conductor.

\section{Preliminaries}

\subsection{Characters}

For a positive integer $d$, we define a character on $(\mathbb{Z}/4d\mathbb{Z})^{*}$
via the Jacobi symbol by 
\[
\chi_{d}(n)=\widetilde{\chi}_{n}(d)=\left(\frac{d}{n}\right).
\]
For odd positive integers $n$ and $d$, we have the following quadratic
reciprocity law for the Jacobi-Kronecker symbol (cf. \cite{mollin},
Theorem 4.2.1, page 197). 
\[
\chi_{d}(n)=\left(\frac{d}{n}\right)
=\left(\frac{n}{d}\right)
(-1)^{\frac{d-1}{2}\cdot\frac{n-1}{2}}
=\begin{cases}
\widetilde{\chi}_{d}(n), & d\equiv1\,(\text{mod 4)};\\
\widetilde{\chi}_{d}(-n)=\widetilde{\chi}_{d}(n)\psi_{-1}(n), & d\equiv3\,(\text{mod 4)}.
\end{cases}
\]

\subsection{$L$-function results}

Suppose that $\chi$ is a Dirichlet character. We define 
\begin{equation}
L^{(P)}(s,\chi)=L(s,\chi)\prod_{p\mid P}\left(1-\frac{\chi(p)}{p^{s}}\right).\label{eq:L-remove-primes}
\end{equation}
We define the odd sign indicator function of a Dirichlet character
$\chi$ by $\kappa=\kappa(\chi)=\frac{1}{2}(1-\chi(-1))$.

\subsection{$L$-functional bounds and approximate functional equation}

For a primitive character $\chi$ modulo $q$, using an absolute convergence argument
for $L(s,\chi)$ in a right half-plane and applying the functional
equation, we interpolate via the Phragm\'en-Lindelöf convexity principle
to obtain
\begin{equation}
L(s,\chi)\ll\begin{cases}
[q(1+\vert\Im s\vert)]^{1/2-\Re s}, & \Re s\leq-\varepsilon;\\
{}[q(1+\vert\Im s\vert)]^{(1-\Re s)/2+\varepsilon}, & -\varepsilon<\Re s<1+\varepsilon;\\
1, & \Re s\geq1+\varepsilon,
\end{cases}\label{eq:L-bounds}
\end{equation}
away from a possible pole at $s=1$ in the case where $\chi$ is trivial.
This bound is known as the convexity bound for Dirichlet $L$-functions.

We shall also need the so-called approximate functional equation for
Dirichlet $L$-functions (cf. \cite{IK04}, Theorem 5.3). Particularly,
if $\chi$ is a quadratic primitive character modulo odd $q$, $\psi$
is a character with conductor dividing 8, and $d_{0}$ is odd, squarefree
and coprime to $q$, then we have the weighted infinite sum
\begin{equation}
L({\textstyle \frac{1}{2}},\chi_{d_{0}}\chi\psi)=2\sum_{n=1}^{\infty}\frac{(\chi_{d_{0}}\chi\psi)(n)}{n^{1/2}}G_{\kappa}\left(\frac{n}{\sqrt{c_{0}d_{0}q}}\right),\label{eq:L-approx-functional}
\end{equation}
where
\begin{equation}
\begin{array}{cc}
\kappa=\kappa(\chi_{d_{0}}\chi\psi), & c_{0}=\begin{cases}
1, & d_{0}\equiv1\,(\text{mod }4),\psi=\psi_{1}\,\text{or}\, d_{0}\equiv3\,(\text{mod }4),\psi=\psi_{-1};\\
4, & d_{0}\equiv1\,(\text{mod }4),\psi=\psi_{-1}\,\text{or}\, d_{0}\equiv3\,(\text{mod }4),\psi=\psi_{1};\\
8, & \psi=\psi_{2}\text{ or }\psi_{-2},
\end{cases}\end{array}\label{eq:Lfunctional-kappa,delta0-1}
\end{equation}
and we have the weight function 
\begin{equation}
G_{\kappa}(\xi)=\frac{1}{2\pi i}\int_{(2)}\frac{\Gamma(\frac{1/2+s+\kappa}{2})}{\Gamma(\frac{1/2+\kappa}{2})}\xi{}^{-s}\frac{ds}{s},\label{eq:G-def}
\end{equation}
satisfying the bound
\begin{equation}
\label{eq:G-bound}
G_{\kappa}(\xi) \ll (1+\xi)^{-A}
\end{equation}
for arbitrary $A\geq0$ (cf.\ \cite{IK04}, Proposition 5.4).

We shall also make use of smooth weight functions.
\begin{defn}
\label{def:smooth-weight-function}We say that $w(x)$ is a smooth
weight function if it is a smooth non-negative real function supported
on $[1/4,5/4]$ and unity on $[1/2,1]$.
\end{defn}
The following bound can be shown via sufficiently many applications
of integration by parts. We have
\begin{equation}
\tilde{w}(z)\ll_{A,\Re z}(1+\vert z\vert)^{-A}\label{eq:w-bound}
\end{equation}
for $A\geq0$, where $\tilde{w}(z)$ is the Mellin transform of $w(x)$.

\subsection{Short double character sums and $L$-function moments}

We shall need the following adaptation of Theorem 2 from \cite{hb}
which includes a character twist.
\begin{thm}
\label{thm:HB-fouth-moment}Let $\psi$ be a primitive character modulo
$j$, and for a positive integer $Q$ define $S(Q)$ to be the set
of quadratic primitive Dirichlet characters of conductor at most $Q$.
We have
\[
\sum_{\chi\in S(Q)}\vert L(\sigma+it,\chi\psi)\vert^{4}\ll_{\varepsilon}\lbrace Q+(Qj(\vert t\vert+1))^{2-2\sigma}\rbrace\lbrace Qj(\vert t\vert+1)\rbrace^{\varepsilon}
\]
for any fixed $\sigma\in[1/2,1]$ and any $\varepsilon>0$.
\end{thm}
Removing even conductors, and using the fact that $2-2\sigma\leq1$,
we can restate this as follows.
\begin{thm}
\label{thm:heath-brown-estimate-fourth}If $\chi$ is a primitive
character with conductor $q$ and $X$ is a positive real number,
then 
\[
\sum_{\substack{d_{0}\leq X\\
d_{0}\text{ odd, squarefree}
}
}\vert L(s,\chi_{d_{0}}\chi)\vert^{4}\ll_{\varepsilon}(Xq\vert s\vert)^{1+\varepsilon},\ \sigma\geq{\textstyle \frac{1}{2}}
\]
for all $\varepsilon>0$.
\end{thm}
An important ingredient for the proof of the subconvexity bound Theorem
\ref{thm:subconvexity} is a large sieve estimate for quadratic characters
due to Heath-Brown. In particular we state here Theorem 1 in \cite{hb}.
\begin{thm}[Heath-Brown's large sieve estimate]
\label{thm:HeathBrown-Theorem-1} Let $P$ and $Q$ be positive integers,
and let $(a_{n})$ be a sequence of complex numbers. Then
\[
\left.\sum_{m\leq P}\right.^{\ast}\left|\left.\sum_{n\leq Q}\right.^{\ast}a_{n}\left(\frac{n}{m}\right)\right|^{2}\ll_{\varepsilon}(PQ)^{\varepsilon}(P+Q)\left.\sum_{n\leq Q}\right.^{\ast}\vert a_{n}\vert^{2},
\]
for any $\varepsilon>0$, where $\left.\sum\right.^{\ast}$ denotes
that the sum is over odd squarefree numbers.
\end{thm}
Due to the nature of the double Dirichlet series we shall construct,
we shall need the following normalization of this result.
\begin{cor}
\label{cor:heath-brown-large-sieve}If $(a_{m}),\ (b_{n})$ are sequences
of complex numbers satisfying the bound $a_{m},\ b_{m}\ll m^{-1/2+\varepsilon}$
for some $\varepsilon>0$, and $P$ and $Q$ are positive real numbers,
then
\[
\sum_{\substack{m\leq P\\
m\text{ odd}
}
}\,\sum_{\substack{n\leq Q\\
n\text{ odd}
}
}a_{m}b_{n}\left(\frac{n_{0}}{m}\right)\ll_{\varepsilon}(PQ)^{\varepsilon}(P+Q)^{1/2+\varepsilon},
\]
where we have the composition $n=n_{0}n_{1}^{2}$ with $n_{0}$ squarefree,
uniformly in $P$ and $Q$, for any $\varepsilon>0$.

\end{cor}

\subsection{Gamma identities}

We shall have use for the identity 
\begin{equation}
\frac{\Gamma(\frac{2-z}{2})}{\Gamma(\frac{z+1}{2})}=\frac{\Gamma(\frac{1-z}{2})}{\Gamma(\frac{z}{2})}\cot\left(\frac{\pi z}{2}\right),\ z\in\mathbb{C}.\label{eq:cot-identity}
\end{equation}
By Stirling's formula, in particular (5.113) from \cite{IK04},
for $s\in\mathbb{C}$ with fixed real part and nonzero imaginary part,
we have 
\begin{equation}
\frac{\Gamma(\frac{1-s}{2})}{\Gamma(\frac{s}{2})}\ll_{\Re s}(1+\vert s\vert)^{1/2-\Re s},\label{eq:stirling}
\end{equation}
away from the poles at the odd positive integers. We also have the
cotangent bound,
\begin{equation}
\cot(x+iy)=-i\,\text{sign}(y)+O(e^{-2\vert y\vert}),\ \min_{k\in\mathbb{Z}}\vert z-\pi k\vert\geq1/10.\label{eq:cot-bound}
\end{equation}

\section{Structure and Analytic Properties}

\subsection{Switch of summation formula}

The object we would like to study is
\[
Z_{0}(s,w;\chi,\chi')=\sum_{d\geq1}\frac{L(s,\chi_{d}\chi)\chi'(d)}{d^{w}},
\]
where $\chi$ and $\chi'$ are quadratic characters with moduli $N$
and $M$ respectively. However, in order to obtain functional equations,
we will need to augment this by some correction factors. The exact
form of these correction factors (in a more general setting) was determined
in \cite{bfh}. Applying this theory in our case, we have the following
theorem, which we note holds for general Dirichlet characters $\chi$
and $\chi'$, though in our case we are interested in the special
case of quadratic twists.
\begin{thm}
\label{thm:correction-factors}Let $m$ and $d$ be positive integers
with $(md,2MN)=1$, and write $d=d_{0}d_{1}^{2}$ and $m=m_{0}m_{1}^{2}$,
with $d_{0}$, $m_{0}$ squarefree and $d_{1}$, $m_{1}$ positive,
and let $\chi$ and $\chi'$ be characters modulo $8\text{lcm}(M,N)$.
Then the Dirichlet polynomials
\[
P_{d_{0},d_{1}}^{(\chi)}(s)=\prod_{p^{\alpha}\Vert d_{1}}\left[\sum_{n=0}^{\alpha}\chi(p^{2n})p^{n-2ns}-\sum_{n=0}^{\alpha-1}(\chi_{d_{0}}\chi)(p^{2n+1})p^{n-(2n+1)s}\right],
\]
\[
Q_{m_{0},m_{1}}^{(\chi')}(w)=\prod_{p^{\beta}\Vert m_{1}}\left[\sum_{n=0}^{\beta}\chi'(p^{2n})p^{n-2nw}-\sum_{n=0}^{\beta}(\tilde{\chi}_{m_{0}}\chi')(p^{2n+1})p^{n-(2n+1)w}\right]
\]
satisfy the functional equations
\begin{equation}
\label{eq:Building-block-reflective}
\begin{array}{ccc}
P_{d_{0},d_{1}}^{(\chi)}(s)=d_{1}^{1-2s}\chi(d_{1}^{2})P_{d_{0},d_{1}}^{(\bar{\chi})}(1-s)
&
\text{and}
&
Q_{m_{0},m_{1}}^{(\chi')}(w)=m_{1}^{1-2w}\chi'(m_{1}^{2})Q_{m_{0},m_{1}}^{(\bar{\chi}')}(1-w),
\end{array}
\end{equation}
and the interchange of summation formula
\begin{equation}
\sum_{(d,2MN)=1}\frac{L^{(2MN)}(s,\chi_{d_{0}}\chi)\chi'(d)P_{d_{0},d_{1}}^{(\chi)}(s)}{d^{w}}=\sum_{(m,2MN)=1}\frac{L^{(2MN)}(w,\tilde{\chi}_{m_{0}}\chi')\chi(m)Q_{m_{0},m_{1}}^{(\chi')}(w)}{m^{s}},\label{eq:Building-block-sum-switch}
\end{equation}
for $\Re s,\Re w>1$.\end{thm}
\begin{proof}
This is a straightforward but lengthy computation which we supress
for brevity. We direct the reader to \cite{bfh} for details.
\end{proof}
In light of the above theorem, we now define our double Dirichlet
series as follows: Let $\chi$ and $\chi'$ be characters modulo $8\text{lcm}(M,N)$.
Then define
\begin{equation}
Z(s,w;\chi,\chi')=\sum_{(d,2MN)=1}\frac{L^{(2MN)}(s,\chi_{d_{0}}\chi)\chi'(d_{0})P_{d_{0},d_{1}}^{(\chi)}(s)}{d^{w}}.\label{eq:Z-definition}
\end{equation}

We note at this point that it is easily shown that, for $d=d_{0}d_{1}^{2}$
and $m=m_{0}m_{1}^{2}$ with $\Re s,\Re w\geq{\textstyle \frac{1}{2}}$,
we have the bounds $\vert P_{d_{0},d_{1}}^{(\chi)}(s)\vert\ll d_{1}^{\varepsilon}$
and $\vert Q_{m_{0},m_{1}}^{(\chi')}(w)\vert\ll m_{1}^{\varepsilon}$.
Applying the functional equations (\ref{eq:Building-block-reflective}), we therefore have 
\begin{equation}
\begin{array}{cccc}
\vert P_{d_{0},d_{1}}^{(\chi)}(s)\vert\ll\begin{cases}
d_{1}^{1-2\Re s+\varepsilon}, & \Re s<{\textstyle \frac{1}{2}};\\
d_{1}^{\varepsilon}, & \Re s\geq{\textstyle \frac{1}{2}},
\end{cases} &  &  & \vert Q_{m_{0},m_{1}}^{(\chi')}(w)\vert\ll\begin{cases}
m_{1}^{1-2\Re w+\varepsilon}, & \Re w<{\textstyle \frac{1}{2}};\\
m_{1}^{\varepsilon}, & \Re w\geq{\textstyle \frac{1}{2}}.
\end{cases}\end{array}\label{eq:P,Q-bounds}
\end{equation}

\subsection{Functional equations}

We recall that $M$ and $N$ are odd prime numbers or unity, possibly
equal. From this point on in the paper, we shall use the following
notation: Let $\chi$ and $\chi'$ be quadratic primitive characters
of squarefree conductors $k$ and $j$ respectively, where $j,k\mid\text{lcm}(M,N)$,
and let $\psi,\ \psi'$ be primitive characters with conductors dividing
8. We first derive the following expansion of the regions of absolute
convergence of the key series involved.
\begin{lem}
\label{lem:Z-initial-continuation}We have the following two series
representations of $Z(s,w;\chi\psi,\chi'\psi')$. We have
\[
Z(s,w;\chi\psi,\chi'\psi')=\sum_{(d,2MN)=1}\frac{L^{(2MN)}(s,\chi_{d_{0}}\chi\psi)\chi'\psi'(d)P_{d_{0},d_{1}}^{(\chi\psi)}(s)}{d^{w}}
\]
which is absolutely convergent on the set
\[
R_{1}^{(1)}:=\{\Re s\leq0,\Re w+\Re s>3/2\}\cup\{0<\Re s\leq1,\Re s/2+\Re w>3/2\}\cup\{\Re s,\Re w>1\},
\]
except for a possible polar line $\{s=1\}$, and 
\[
Z(s,w;\chi\psi,\chi'\psi')=\sum_{(m,2MN)=1}\frac{L^{(2MN)}(w,\tilde{\chi}_{m_{0}}\chi'\psi')\chi\psi(m)Q_{m_{0},m_{1}}^{(\chi'\psi')}(w)}{m^{s}}
\]
which is absolutely convergent on the set
\[
R_{1}^{(2)}:=\{\Re w\leq0,\Re s+\Re w>3/2\}\cup\{0<\Re w\leq1,\Re w/2+\Re s>3/2\}\cup\{\Re s,\Re w>1\},
\]
except for a possible polar line $\{w=1\}$.\end{lem}
\begin{proof}
This follows from applying the bounds for the Dirichlet $L$-function (\ref{eq:L-bounds}) and the bounds for the correction polynomials (\ref{eq:P,Q-bounds})
to the definition (\ref{eq:Z-definition}).
\end{proof}
We now proceed with derivation of functional equations for $Z$. Due
to the summation switch formula (\ref{eq:Building-block-sum-switch}),
we have
\[
Z(s,w;\chi\psi,\chi'\psi')=\sum_{(m,2MN)=1}\frac{L^{(2MN)}(w,\tilde{\chi}_{m_{0}}\chi'\psi')(\chi\psi)(m)Q_{m_{0},m_{1}}^{(\chi'\psi')}(w)}{m^{s}}.
\]
We can further apply the functional equation for the $Q$ factor (\ref{eq:Building-block-reflective})
to obtain
\begin{equation}
Z(s,w;\chi\psi,\chi'\psi')=\sum_{(m,2MN)=1}\frac{L^{(2MN)}(w,\tilde{\chi}_{m_{0}}\chi'\psi')(\chi\psi)(m)\chi'(m_{1}^{2})Q_{m_{0},m_{1}}^{(\chi'\psi')}(1-w)}{m^{s}m_{1}^{2(w-1/2)}},\label{eq:Z-transformed-first-functional}
\end{equation}
which holds for $(s,w)\in R_{1}^{(2)}$. The next step is to apply
the functional equation for Dirichlet $L$-functions in order to change
the $w$ in the $L$-function to $1-w$. This would allow us to switch
summation again to obtain $Z$ in its original form, but with a change
in variables.

We shall define the following Euler product function: For a character
$\chi^{\star}$ and a positive integer $P$, we define
\begin{equation}
K_{P}(w;\chi^{\star})=\prod_{p\mid P}\left(1-\frac{\chi^{\star}(p)}{p^{1-w}}\right)^{-1}\left(1-\frac{\chi^{\star}(p)}{p^{w}}\right).\label{eq:K-definition}
\end{equation}
Applying the Dirichlet functional equation along with (\ref{eq:L-remove-primes})
and the above, we now have
\begin{multline}
L^{(2MN)}(w,\tilde{\chi}_{m_{0}}\chi'\psi')=
\pi^{w-\frac{1}{2}}\frac{\Gamma\left(\frac{1-w+\hat{\kappa}'}{2}\right)}{\Gamma\left(\frac{w+\hat{\kappa}'}{2}\right)}K_{2MN}(w;\tilde{\chi}_{m_{0}}\chi'\psi')(C_{\psi'}jm_{0})^{\frac{1}{2}-w}L^{(2MN)}(1-w,\tilde{\chi}_{m_{0}}\chi'\psi'),\label{eq:L-functional-first-functional}
\end{multline}
where $\hat{\kappa}'=\kappa(\tilde{\chi}_{m_{0}}\chi'\psi')$.

We need to break down some of these parts for further manipulation.
Recalling that $(m_{0},2MN)=1$, if $p$ is prime and does not divide
$m_{0}$, we have
\begin{equation}
K_{p}(w;\tilde{\chi}_{m_{0}}\chi^{\star})=\frac{\chi^{\star}(p^{2})p-p^{2}+\chi_{p}(m_{0})\chi^{\star}(p)(p^{2-w}-p^{1+w})}{\chi^{\star}(p^{2})p^{2w}-p^{2}}.\label{eq:K-identity}
\end{equation}
If for $P\in\mathbb{N}$ and a Dirichlet character $\chi^{\star}$
we set

\begin{equation}
\begin{array}{ccc}
F_{P}^{(\chi^{\star})}(w)=\begin{cases}
1, & P=1;\\
\frac{\chi^{\star}(P^{2})P-P^{2}}{\chi^{\star}(P^{2})P^{2w}-P^{2}}, & \text{else},
\end{cases} &  & G_{P}^{(\chi^{\star})}(w)=\begin{cases}
0, & P=1;\\
\frac{\chi^{\star}(P)(P^{2-w}-P^{1+w})}{\chi^{\star}(P^{2})P^{2w}-P^{2}}, & \text{else},
\end{cases}\end{array}\label{eq:F,G-coeff-def}
\end{equation}

then from (\ref{eq:K-identity}) and the definition (\ref{eq:K-definition})
we have the identity
\begin{equation}
K_{p}(w;\tilde{\chi}_{m_{0}}\chi^{\star})=F_{p}^{(\chi^{\star})}(w)+\chi_{p}(m_{0})G_{p}^{(\chi^{\star})}(w),\label{eq:K-identity-2}
\end{equation}
which holds for prime $p$ not dividing $m_{0}$, or $p=1$. Noting
that $K_{P}(w;\chi^{\star})$ is multiplicative in $P$, and using
(\ref{eq:K-identity-2}) above, we now have the useful expression
\begin{multline}
K_{MN}(w;\tilde{\chi}_{m_{0}}\chi'\psi')=F_{M}^{(\chi'\psi')}\cdot F_{N}^{(\chi'\psi')}(w)+\chi_{M}(m_{0})F_{N}^{(\chi'\psi')}\cdot G_{M}^{(\chi'\psi')}(w)\\
+\chi_{N}(m_{0})F_{M}^{(\chi'\psi')}\cdot G_{N}^{(\chi'\psi')}(w)+\chi_{MN}(m_{0})G_{M}^{(\chi'\psi')}\cdot G_{N}^{(\chi'\psi')}(w).\label{eq:K-identity-3}
\end{multline}
We note that this holds true even if $M=N$. Indeed, in this case,
by the definition (\ref{eq:K-definition}), we have $K_{MN}(w;\tilde{\chi}_{m_{0}}\chi'\psi')=1$.
This is consistent with (\ref{eq:K-identity-3}), since according
to definition (\ref{eq:F,G-coeff-def}) we have $G_{N}^{(\chi'\psi')}(w)=0$
and $F_{N}^{(\chi'\psi')}(w)=1$.

Next, we see from (\ref{eq:cot-identity}) that
\begin{equation}
\frac{\Gamma\left(\frac{1-w+\hat{\kappa}'}{2}\right)}{\Gamma\left(\frac{w+\hat{\kappa}'}{2}\right)}=\frac{\Gamma\left(\frac{1-w}{2}\right)}{\Gamma\left(\frac{w}{2}\right)}\cot\left(\frac{\pi w}{2}\right)^{\hat{\kappa}'}.\label{eq:Gamma-quotient-factored}
\end{equation}

We shall find it useful to remove the dependency of $\hat{\kappa}'$
on $m_{0}$, or rather, exploit that the dependency is only on its
residue modulo 4. Hence, define $\kappa'=\kappa(\chi'\psi')$. Now
suppose that $f$ is a function of $\hat{\kappa}'$. By sieving out
by congruence classes modulo 4, we see that
\[
f(\hat{\kappa}')=\frac{1}{2}(1+\psi_{-1}(m_{0}))f(\kappa')+\frac{1}{2}(1-\psi_{-1}(m_{0}))f(1-\kappa').
\]
Hence we have
\begin{equation}
\cot\left(\frac{\pi w}{2}\right)^{\hat{\kappa}'}=\frac{1}{2}(1+\psi_{-1}(m_{0}))\cot\left(\frac{\pi w}{2}\right)^{\kappa'}+\frac{1}{2}(1-\psi_{-1}(m_{0}))\cot\left(\frac{\pi w}{2}\right)^{(1-\kappa')}.\label{eq:cot-sieved}
\end{equation}
For brevity, for a character $\chi^{\star}$, we define
\[
S(s,w;m,\chi^{\star}):=\frac{L^{(2MN)}(1-w,\tilde{\chi}_{m_{0}}\chi'\psi')\chi^{\star}(m)Q_{m_{0},m_{1}}^{(\chi'\psi')}(1-w)}{m^{s+w-1/2}}.
\]
Applying the functional equation for Dirichlet $L$-functions (\ref{eq:L-functional-first-functional})
to the identity (\ref{eq:Z-transformed-first-functional}) along with
(\ref{eq:cot-sieved}), and using Lemma \ref{lem:Z-initial-continuation}
we now get
\begin{multline}
Z(s,w;\chi\psi,\chi'\psi')=\frac{1}{2}\pi^{w-1/2}\frac{\Gamma\left(\frac{1-w}{2}\right)}{\Gamma\left(\frac{w}{2}\right)}(jC_{\psi'})^{1/2-w}\sum_{(m,2MN)=1}K_{2MN}(w;\tilde{\chi}_{m_{0}}\chi'\psi')\\
\times S(s,w;m,\chi\psi)\left[(1+\psi_{-1}(m_{0}))\cot\left(\frac{\pi w}{2}\right)^{\kappa'}+(1-\psi_{-1}(m_{0}))\cot\left(\frac{\pi w}{2}\right)^{(1-\kappa')}\right],\label{eq:Z-identity1}
\end{multline}
for $(s,w)\in R_{1}^{(2)}$ except for possible polar lines $\{s=1\}$
and $\{w=1\}$. There are two properties of the function $S$ which
we shall make use of, stated in the following lemma.
\begin{lem}
\label{prop:S}For two characters $\chi^{\star}$ and $\chi^{\star\star}$
and an integer $m$, the following two properties hold for $(s,w)\in R_{1}^{(2)}$,
except for possible polar lines $\{s=1\}$ and $\{w=1\}$.

\[
\begin{array}{rrcl}
\text{(i)} & S(s,w;m,\chi^{\star})\chi^{\star\star}(m) & = & S(s,w;m,\chi^{\star}\chi^{\star\star}),\\
\text{(ii)} & {\displaystyle \sum_{(m,2MN)=1}}S(s,w;m,\chi^{\star}) & = & Z(s+w-{\textstyle \frac{1}{2}},1-w;\chi^{\star},\chi'\psi').
\end{array}
\]
\end{lem}
We can apply the identity (\ref{eq:K-identity-3}) and Lemma \ref{prop:S}
to (\ref{eq:Z-identity1}), and use the bounds (\ref{eq:stirling})
and (\ref{eq:cot-bound}) to come to the following functional equation.
\begin{thm}
\label{thm:functional-1}Let $\chi$ and $\chi'$ be primitive Dirichlet
characters modulo squarefree $k$ and $j$ respectively, where $j,k\mid MN$,
and if $M=N$, then $j=k=M=N$, and let $\psi$ and $\psi'$ be Dirichlet
characters modulo 8. There exist functions $a_{n}^{(\chi',\psi')}(w;\psi^{\star})$
for $n\mid MN$ and $\psi^{\star}$ a Dirichlet character modulo 8
which are holomorphic except for possible poles at the positive integers,
and countably many poles on the line $\Re w=1$, bounded absolutely above by
$O((16\pi)^{\vert\Re w\vert}(1+\vert w\vert)^{1/2-\Re w})$ uniformly
in $j$ and $k$ away from the poles such that for $(s,w)\in R_{1}^{(2)}$
away from possible polar lines $\{s=1\}$ and $\{w=1\}$ we have
\begin{multline*}
Z(s,w;\chi\psi,\chi'\psi')=j{}^{1/2-w}\sum_{n\mid MN}A_{n}^{(\chi')}(w)\sum_{\psi^{\star}\in\widehat{(\mathbb{Z}/8\mathbb{Z})^{*}}}a_{n}^{(\chi',\psi')}(w;\psi^{\star})Z(s+w-{\textstyle \frac{1}{2}},1-w;\chi\tilde{\chi}_{n}\psi\psi^{\star},\chi'\psi'),
\end{multline*}
where $R_{1}^{(2)}$ is defined in Lemma \ref{lem:Z-initial-continuation},
we have

\[
\begin{array}{ccc}
A_{1}^{(\chi')}(w)=F_{M}^{(\chi')}(w)F_{N}^{(\chi')}(w), &  & A_{M}^{(\chi')}(w)=F_{N}^{(\chi')}(w)G_{M}^{(\chi')}(w),\\
A_{N}^{(\chi')}(w)=F_{M}^{(\chi')}(w)G_{N}^{(\chi')}(w), &  & A_{MN}^{(\chi')}(w)=G_{M}^{(\chi')}(w)G_{N}^{(\chi')}(w),
\end{array}
\]

and the $F$ and $G$ functions are defined in (\ref{eq:F,G-coeff-def}).
\end{thm}
By similar methods, we can obtain a second functional equation under
the transformation $(s,w)\mapsto(1-s,s+w-{\textstyle \frac{1}{2}})$
by applying the functional equation (\ref{eq:Building-block-reflective})
to the definition (\ref{eq:Z-definition}), followed by the functional
equation for $L$-functions. The result of this similarly lengthy
derivation is the following second functional equation.
\begin{thm}
\label{thm:functional-2}Let $\chi$ and $\chi'$ be Dirichlet characters
modulo squarefree $k$ and $j$ respectively, where $j,k\mid MN$,
and if $M=N$, then $j=k=M=N$, and let $\psi$ and $\psi'$ be Dirichlet
characters modulo 8. There exist functions $b_{n}^{(\chi,\psi)}(s;\psi^{\star})$
for $n\mid MN$ and $\psi^{\star}$ a Dirichlet character modulo 8
which are holomorphic except for possible poles at the positive integers,
and countably many poles on the line $\Re s=1$, bounded absolutely above by
$O((16\pi)^{\vert\Re s\vert}(1+\vert s\vert)^{1/2-\Re s})$ uniformly
in $j$ and $k$ away from the poles such that for $(s,w)\in R_{1}^{(1)}$
away from possible polar lines $\{s=1\}$ and $\{w=1\}$ we have 
\begin{multline*}
Z(s,w;\chi\psi,\chi'\psi')=k{}^{1/2-s}\sum_{n\mid MN}A_{n}^{(\chi)}(s)\sum_{\psi^{\star}\in\widehat{(\mathbb{Z}/8\mathbb{Z})^{*}}}b_{n}^{(\chi,\psi)}(s;\psi^{\star})Z(1-s,s+w-{\textstyle \frac{1}{2}};\chi\psi,\chi'\tilde{\chi}_n\psi'\psi^{\star}),
\end{multline*}
where $R_{1}^{(1)}$ is defined in Lemma \ref{lem:Z-initial-continuation},
and the $A$ functions are as in Theorem \ref{thm:functional-1}.\end{thm}

We make note of an analytic subtlety: We note that although the $a_{n}$,
$b_{n}$, and $A_{n}$ functions above have poles, they do not contribute
poles to $Z(s,w;\chi\psi,\chi'\psi')$; indeed, it will be proven
in Proposition \ref{prop:continuation} that the only possible poles
of $Z$ are the polar lines $\{s=1\}$, $\{w=1\}$, and $\{s+w=3/2\}$.
Looking at (\ref{eq:Z-identity1}), though the gamma and cotangent
factors together have poles at either the even or odd positive integers,
and the $K$ factor has countably many poles along the line $\Re w=1$
(cf. (\ref{eq:K-identity})), these poles nonetheless do not produce
poles on the right-hand side. The equation (\ref{eq:Z-identity1})
essentially results from the application of the functional equation
for $L$-functions (\ref{eq:L-functional-first-functional}) to the
identity (\ref{eq:Z-transformed-first-functional}), and subsequently
sieving out by congruence classes of $m$ modulo 4. The last step
introduces coefficients with poles from the gamma and cotangent factors.
This is a manifestation of a phenomenon that is observed in the functional
equation for $L$-functions: In (\ref{eq:L-functional-first-functional}),
we know that the $L$-function can only have a pole at $w=1$, yet
on the right-hand side, the gamma function produces poles which are
mitigated by the trivial zeros of the $L$-function. It is precisely
these poles which appear in the coefficients of (\ref{eq:Z-identity1}).
Additionally, although the $K$ function has poles, these are mitigated
by corresponding zeros of the $L$ function due to removal of the
Euler factors at primes dividing $2MN$.

It shall be useful to note the following properties of the $A$ coefficients,
which follow directly from the definitions.
\begin{lem}
\label{lem:functional-coefficients}Let $\chi$ be a Dirichlet character
modulo $q$. Then the following properties hold.

\begin{enumerate}

\item For a positive integer $P$, if $(q,P)>1$ then $A_{P}^{(\chi)}(w)=0$.

\item If $q=MN$ then $A_{1}^{(\chi)}(w)=1$.

\item If $q=M$ then $A_{1}^{(\chi)}(w)=F_{N}^{(\chi)}(w)$ and $A_{N}^{(\chi)}(w)=G_{N}^{(\chi)}(w)$. 

\item If $q=N$ then $A_{1}^{(\chi)}(w)=F_{M}^{(\chi)}(w)$ and $A_{M}^{(\chi)}(w)=G_{M}^{(\chi)}(w)$.

\item
Moreover, the following asymptotics hold, if $P\neq1$ and $(P,q)=1$.

\[
\begin{array}{ccc}
\vert F_{P}^{(\chi)}(w)\vert\asymp\begin{cases}
1, & \Re w<1-\varepsilon;\\
P^{2-2\Re w}, & \Re w>1+\varepsilon,
\end{cases} &  & \vert G_{P}^{(\chi)}(w)\vert\asymp\begin{cases}
P^{-\Re w}, & \Re w<{\textstyle \frac{1}{2}};\\
P^{\Re w-1}, & {\textstyle \frac{1}{2}}\leq\Re w<1-\varepsilon;\\
P^{1-\Re w}, & \Re w>1+\varepsilon.
\end{cases}\end{array}
\]

\end{enumerate}

\end{lem}

It shall also be useful to derive a somewhat symmetric functional equation, obtained by application of Theorem \ref{thm:functional-1}, then Theorem \ref{thm:functional-2}, followed again by Theorem \ref{thm:functional-1}. For quadratic Dirichlet characters $\rho$ and $\rho'$ of conductors $N$ and $M$ respectively which are either prime or unity, possibly equal, we have

\begin{multline}
\label{eq:functional-3}
Z(s,w;\rho\psi,\rho'\psi')=
M^{1/2-w}
\sum_{n,m,r\mid MN}
C_{\rho \tilde{\chi}_n}^{1-s-w}
C_{\rho' \tilde{\chi}_m}^{1/2-s}
	A_{n}^{(\rho')}(w)
	A_m^{(\rho \tilde{\chi}_n)}(s+w-1/2)
	A_r^{(\rho'\tilde{\chi}_m)}(w)
\\
\times \sum_{\psi^{\star}, \psi^{\star\star}, \psi^{\star\star\star}\in\widehat{(\mathbb{Z}/8\mathbb{Z})^{*}}}
	c_{n,m,r}^{(\rho,\rho',\psi,\psi')}(s,w;\psi^\star,\psi^{\star\star},\psi^{\star\star\star})
Z(1-w,1-s;\rho\tilde{\chi}_{nr}\psi\psi^{\star}\psi^{\star\star\star},\rho'\tilde{\chi}_m\psi'\psi^{\star\star}),
\end{multline}
where the $A$ functions are as in Theorem \ref{thm:functional-1}, the $c$ functions are holomorphic in $\mathbb{C}^2$ except for possible poles for $s$, $w$, or $s+w-1/2$ equal to positive integers, and countably many poles on the lines $\Re w = 1$, $\Re s = 1$, and $\Re s + \Re w = 3/2$, bounded absolutely above by
\begin{equation}
\label{eq:functional-3-coeffs-bound}
O((16\pi)^{2\vert \Re s \vert+2\vert \Re w \vert} (1+\vert s \vert)^{1/2-\Re s} (1+\vert w \vert)^{1/2-\Re w} (1+\vert s+w \vert)^{1-\Re s - \Re w}).
\end{equation}

\subsection{Analytic continuation}

We continue $Z(s,w;\chi,\chi')$ to all of $\mathbb{C}^{2}$ except
for the polar lines $s=1$, $w=1$, and $s+w=3/2$. We have 
\begin{prop}
\label{prop:continuation}Let $\chi$ and $\chi'$ be characters modulo
$8\text{lcm}(M,N)$. The function 
\[
\tilde{Z}(s,w;\chi,\chi')=(s-1)(w-1)(s+w-3/2)Z(s,w;\chi,\chi')
\]
 is holomorphic in $\mathbb{C}^{2}$ and is polynomially bounded in
the sense that, given $C_{1}>0$, there exists $C_{2}>0$ such that
$\tilde{Z}(s,w;\chi,\chi')\ll[MN(1+\vert\Im s\vert)(1+\vert\Im w\vert)]^{C_{2}}$
whenever $\vert\Re s\vert,\vert\Re w\vert<C_{1}$.\end{prop}
\begin{proof}
We refer to the reader to \cite{bfh} and \cite{blomer}.
\end{proof}

\subsection{\label{sub:convexity-bound}Convexity bound}

The notion of convexity is not canonically defined for double Dirichlet
series as it is in the case of (single) Dirichlet series; in the latter
case, we have a single functional equation which reflects the region
of absolute convergence, and interpolating the bounds produces a convexity
bound between the two. In the case of double Dirichlet series, things
are more complicated. Firstly, our bounds in the region of absolute
convergence depend on our knowledge of the bounds on $L(s,\chi)$
on average. Secondly, we have 6 functional equations to choose from
to apply to this region. If we use the Lindelöf hypothesis on average,
namely Theorem \ref{thm:heath-brown-estimate-fourth}, then we can
carefully choose a functional equation to apply in order to minimize
the resulting convexity bound from application of the Phragm\'en-Lindelöf
convexity principle (cf. Theorem 5.53 of \cite{IK04}).

We shall require some initial bounds. We let $\chi$ and $\chi'$
be quadratic Dirichlet characters with conductors $k$ and $j$ respectively,
and $\psi,\psi'$ are characters modulo 8. We first assume that $\Re s=1/2$
and $\Re w=1+\varepsilon$. We apply the identity (\ref{eq:L-remove-primes}),
the bounds (\ref{eq:P,Q-bounds}), and the Cauchy-Schwarz inequality
with the average bound of Theorem \ref{thm:heath-brown-estimate-fourth}
to obtain 
\begin{equation}
Z(s,w;\chi\psi,\chi'\psi')\ll(MN)^{\varepsilon}k^{1/4+\varepsilon}(1+\vert s\vert)^{1/4+\varepsilon},\Re s=1/2,\Re w=1+\varepsilon,\label{eq:convex-trivial-1}
\end{equation}
and by the switch of summation formula (\ref{eq:Building-block-sum-switch})
we also obtain
\begin{equation}
Z(s,w;\chi\psi,\chi'\psi')\ll(MN)^{\varepsilon}j^{1/4+\varepsilon}(1+\vert w\vert)^{1/4+\varepsilon},\ \Re s=1+\varepsilon,\Re w=1/2.\label{eq:convex-trivial-2}
\end{equation}

We use the functional equation Theorem \ref{thm:functional-2} with
$\Re s=-\varepsilon$ and $\Re w=1+\varepsilon$ and apply (\ref{eq:convex-trivial-2})
on the right-hand side in order to obtain a bound for $Z(s,w;\rho,\rho')$.
Looking at the coefficient bounds in Lemma \ref{lem:functional-coefficients},
we pick up a factor of $N^{1/2+\varepsilon}$. Further, we see that
the resulting twisting characters on the right-hand side will have
conductors $(k,j)\in\{(N,M),(N,1)\}$, so that we have 
\begin{equation}
Z(s,w;\rho,\rho')\ll N^{1/2+\varepsilon}M^{1/4+\varepsilon}(1+\vert s\vert)^{1/2+\varepsilon}(1+\vert s+w\vert)^{1/4+\varepsilon},\ \Re s=-\varepsilon,\Re w=1+\varepsilon.\label{eq:pre-interpolate-1}
\end{equation}
Likewise with the functional equation Theorem \ref{thm:functional-1}
applied to (\ref{eq:convex-trivial-1}), we have the symmetric bound
\begin{equation}
Z(s,w;\rho,\rho')\ll M^{1/2+\varepsilon}N^{1/4+\varepsilon}(1+\vert w\vert)^{1/2+\varepsilon}(1+\vert s+w\vert)^{1/4+\varepsilon},\ \Re s=1+\varepsilon,\Re w=-\varepsilon.\label{eq:pre-interpolate-2}
\end{equation}
We wish to interpolate convexly between these bounds along the two
diagonal lines in $\mathbb{C}^{2}$, but we must deal with the potential
poles at $s=1$ and $w=1$. To do this we can simply multiply both
sides of the bounds by $(s-1)(w-1)$. We obtain the following bound.
\begin{prop}
For quadratic characters $\rho$ and $\rho'$ of prime moduli $N$
and $M$ respectively, we have
\[
Z(s,w;\rho,\rho')\ll[(1+\vert s\vert)(1+\vert w\vert)(1+\vert s+w\vert)]^{1/4+\varepsilon}(MN)^{3/8+\varepsilon},\ \Re s=\Re w=1/2,
\]
which we call the convexity bound for the function $Z(s,w;\rho,\rho')$.
\end{prop}

\section{Approximate Functional Equations}

\subsection{A symmetric functional equation}

We introduce a succession of applications of the functional equations
in the special case of $(s,w)=(1/2,1/2-z)$ for some $z\in\mathbb{C}$
with $\Re z>0$. We recall that $\rho$ and $\rho'$ are primitive
quadratic Dirichlet characters of conductors $N$ and $M$ respectively.
We first apply Theorem \ref{thm:functional-1}, which after observing
the coefficient properties of Lemma \ref{lem:functional-coefficients}
gives
\begin{multline}
Z({\textstyle \frac{1}{2}},{\textstyle \frac{1}{2}}-z;\rho,\rho')=\sum_{\psi^{\star}\in\widehat{(\mathbb{Z}/8\mathbb{Z})^{*}}}\left[M^{z}F_{N}^{(\rho')}a_{1}^{(\rho',\psi_{1})}({\textstyle \frac{1}{2}}-z;\psi^{\star})Z({\textstyle \frac{1}{2}}-z,{\textstyle \frac{1}{2}}+z;\rho\psi^{\star},\rho')\right.\\
\left.+M^{z}G_{N}^{(\rho')}({\textstyle \frac{1}{2}}-z)a_{N}^{(\rho',\psi_{1})}({\textstyle \frac{1}{2}}-z;\psi^{\star})Z({\textstyle \frac{1}{2}}-z,{\textstyle \frac{1}{2}}+z;\psi^{\star},\rho')\right].\label{eq:special-functional-pre}
\end{multline}
We then apply the functional equation Theorem \ref{thm:functional-2}
(and again use Lemma \ref{lem:functional-coefficients}) which further
gives
\begin{multline}
Z({\textstyle \frac{1}{2}},{\textstyle \frac{1}{2}}-z;\rho,\rho')=\\
\sum_{\psi^{\star},\psi^{\star\star}\in\widehat{(\mathbb{Z}/8\mathbb{Z})^{*}}}\left[(MN)^{z}F_{N}^{(\rho')}({\textstyle \frac{1}{2}}-z)F_{M}^{(\rho)}({\textstyle \frac{1}{2}}-z)c_{1,1}^{(\rho',\rho)}(z;\psi^{\star},\psi^{\star\star})Z({\textstyle \frac{1}{2}}+z,{\textstyle \frac{1}{2}};\rho\psi^{\star},\rho'\psi^{\star\star})\right.\\
+(MN)^{z}F_{N}^{(\rho')}({\textstyle \frac{1}{2}}-z)G_{M}^{(\rho)}({\textstyle \frac{1}{2}}-z)c_{1,M}^{(\rho',\rho)}(z;\psi^{\star},\psi^{\star\star})Z({\textstyle \frac{1}{2}}+z,{\textstyle \frac{1}{2}};\rho\psi^{\star},\psi^{\star\star})\\
+M^{z}G_{N}^{(\rho')}({\textstyle \frac{1}{2}}-z)F_{M}^{(\psi_{1})}({\textstyle \frac{1}{2}}-z)F_{N}^{(\psi_{1})}({\textstyle \frac{1}{2}}-z)c_{N,1}^{(\rho',\psi_{1})}(z;\psi^{\star},\psi^{\star\star})Z({\textstyle \frac{1}{2}}+z,{\textstyle \frac{1}{2}};\psi^{\star},\rho'\psi^{\star\star})\\
+M^{z}G_{N}^{(\rho')}({\textstyle \frac{1}{2}}-z)F_{N}^{(\psi_{1})}({\textstyle \frac{1}{2}}-z)G_{M}^{(\psi_{1})}({\textstyle \frac{1}{2}}-z)c_{N,M}^{(\rho',\psi_{1})}(z;\psi^{\star},\psi^{\star\star})Z({\textstyle \frac{1}{2}}+z,{\textstyle \frac{1}{2}};\psi^{\star},\psi^{\star\star})\\
+M^{z}G_{N}^{(\rho')}({\textstyle \frac{1}{2}}-z)F_{M}^{(\psi_{1})}({\textstyle \frac{1}{2}}-z)G_{N}^{(\psi_{1})}({\textstyle \frac{1}{2}}-z)c_{N,N}^{(\rho',\psi_{1})}(z;\psi^{\star},\psi^{\star\star})Z({\textstyle \frac{1}{2}}+z,{\textstyle \frac{1}{2}};\psi^{\star},\rho'\rho\psi^{\star\star})\\
\left.+M^{z}G_{N}^{(\rho')}({\textstyle \frac{1}{2}}-z)G_{M}^{(\psi_{1})}({\textstyle \frac{1}{2}}-z)G_{N}^{(\psi_{1})}({\textstyle \frac{1}{2}}-z)c_{N,MN}^{(\rho',\psi_{1})}(z;\psi^{\star},\psi^{\star\star})Z({\textstyle \frac{1}{2}}+z,{\textstyle \frac{1}{2}};\psi^{\star},\rho\psi^{\star\star})\right],\label{eq:special-functional}
\end{multline}
where
\[
c_{n,m}^{(\chi,\chi')}(z;\psi^{\star},\psi^{\star\star})=a_{n}^{(\chi,\psi_{1})}({\textstyle \frac{1}{2}}-z;\psi^{\star})b_{m}^{(\chi',\psi^{\star})}({\textstyle \frac{1}{2}}-z;\psi^{\star\star}).
\]
We note that, in the case where $M=N\neq1$, we may not apply the
functional equation Theorem \ref{thm:functional-2} to $Z({\textstyle \frac{1}{2}}-z,{\textstyle \frac{1}{2}}+z;\psi^{\star},\rho')$
in (\ref{eq:special-functional-pre}) because Theorem \ref{thm:functional-2}
is only valid if $j=k=M=N$, but here we have $k=1\neq N$. Nonetheless,
equation (\ref{eq:special-functional}) still holds: Indeed, in the
case where $M=N$, the term with $Z({\textstyle \frac{1}{2}}-z,{\textstyle \frac{1}{2}}+z;\psi^{\star},\rho')$
of (\ref{eq:special-functional-pre}) vanishes because $G_{N}^{(\rho')}({\textstyle \frac{1}{2}-z)=0}$,
and so only the first term of (\ref{eq:special-functional}) will
remain (note also that $G_{M}^{(\rho)}({\textstyle \frac{1}{2}}-z)=0$
in this case).

Looking at (\ref{eq:F,G-coeff-def}), for a quadratic character $\chi^{\star}$
whose modulus is coprime to $P$, we see that
\begin{equation}
\label{eq:G-coeff-special}
\begin{array}{ccc}
{\displaystyle
F_{P}^{(\chi^{\star})}({\textstyle \frac{1}{2}}-z) = \frac{P^{-1}-1}{P^{-2z-1}-1}
}
&
\text{and}
&
{\displaystyle
G_{P}^{(\chi^{\star})}({\textstyle \frac{1}{2}}-z) = P^{z-1/2}\left(\frac{\chi^{\star}(P)(1-P^{-2z})}{P^{-2z-1}-1}\right).
}
\end{array}
\end{equation}
Therefore, setting
\begin{equation}
\Phi:=\lbrace(\rho,\rho'),\ (\rho,\psi_{1}),\ (\psi_{1},\rho'),\ (\psi_{1},\psi_{1}),\ (\psi_{1},\rho'\rho),\ (\psi_{1},\rho)\rbrace\label{eq:Phi-def}
\end{equation}
we have
\begin{equation}
Z({\textstyle \frac{1}{2}},{\textstyle \frac{1}{2}}-z;\rho,\rho')=\sum_{\substack{(\chi,\chi')\in\Phi\\
\psi,\psi'\in\widehat{(\mathbb{Z}/8\mathbb{Z})^{*}}
}
}\beta_{\psi,\psi'}^{(\chi,\chi')}\omega_{\psi,\psi'}^{(\chi,\chi')}(z)(\gamma_{\psi,\psi'}^{(\chi,\chi')})^{z}Z({\textstyle \frac{1}{2}}+z,{\textstyle \frac{1}{2}};\chi\psi,\chi'\psi'),\label{eq:functional-with-z}
\end{equation}
where we absorb the $F_{P}^{\chi^{\star}}({\textstyle \frac{1}{2}}-z)$
factors and parenthetical expression of (\ref{eq:G-coeff-special})
for the $G_{P}^{\chi^{\star}}({\textstyle \frac{1}{2}}-z)$ factors,
as well as the $c_{n,m}^{(\chi,\chi')}(z;\psi^{\star},\psi^{\star\star})$
factors into the $\omega_{\psi,\psi'}^{(\chi,\chi')}(z)$ functions,
and collect the remaining factors into the $\beta_{\psi,\psi'}^{(\chi,\chi')}$
and $\gamma_{\psi,\psi'}^{(\chi,\chi')}$ coefficients. Hence we see
that for $\Re z>0$, the $\omega_{\psi,\psi'}^{(\chi,\chi')}(z)$
functions are holomorphic satisfying the bound
\begin{equation}
\omega_{\psi,\psi'}^{(\chi,\chi')}(z)\ll\left(1+\vert\Im z\vert\right)^{\Re z}\label{eq:omega-bound}
\end{equation}
uniformly in $M$ and $N$. Thus, we obtain the upper bounds in Table \ref{tab:coeffs}.

\begin{table}[H]
\protect\caption{\label{tab:coeffs}Coefficient upper bounds}

\begin{tabular}{|c|c|c|c|c|c|c|c|}
\hline 
& $(\chi,\chi')$ & $(\rho,\rho')$ & $(\rho,\psi_{1})$ & $(\psi_{1},\rho')$ & $(\psi_{1},\psi_{1})$ & $(\psi_{1},\rho'\rho)$ & $(\psi_{1},\rho)$
\tabularnewline
\hline 
\multirow{2}{*}{$M\neq N$}
& $\beta_{\psi,\psi'}^{(\chi,\chi')}$ bound
& 1 & $M^{-1/2}$ & $N^{-1/2}$ & $M^{-1/2} N^{-1/2}$ & $N^{-1}$ & $M^{-1/2} N^{-1}$
\tabularnewline
\cline{2-8} 
& $\gamma_{\psi,\psi'}^{(\chi,\chi')}$ bound
& $MN$ & $M^2 N$ & $MN$ & $M^2 N$ & $MN^2$ & $M^2 N^2$
\tabularnewline
\hline 
\multirow{2}{*}{$M = N$}
& $\beta_{\psi,\psi'}^{(\chi,\chi')}$ bound
& 1 & 0 & 0 & 0 & 0 & 0
\tabularnewline
\cline{2-8} 
& $\gamma_{\psi,\psi'}^{(\chi,\chi')}$ bound
& $N^2$ & 0 & 0 & 0 & 0 & 0
\tabularnewline
\hline 
\end{tabular}
\end{table}

\subsection{Approximate functional equations}

The following lemma essentially takes the preceding functional equation
a step further by opening the first sum of $Z$.
\begin{lem}
\label{lem:approx-functional}There exist smooth, rapidly decaying
functions $V(\xi;t)$ and $V_{\psi,\psi'}^{(\chi,\chi')}(\xi;t)$
such that for any constant $X>0$ one has
\begin{multline*}
Z({\textstyle \frac{1}{2}},{\textstyle \frac{1}{2}}+it;\rho,\rho')=X^{-it}\sum_{\substack{(d,2MN)=1}
}\frac{L^{(2MN)}(\frac{1}{2},\chi_{d_{0}}\rho)\rho'(d)P_{d_{0},d_{1}}^{(\rho)}(\frac{1}{2})}{d{}^{1/2}}V\left(\frac{d}{X};t\right)\\
+X^{it}\sum_{\substack{(\chi,\chi')\in\Phi\\
\psi,\psi'\in\widehat{(\mathbb{Z}/8\mathbb{Z})^{*}}
}
}\beta_{\psi,\psi'}^{(\chi,\chi')}\sum_{\substack{(d,2MN)=1}
}\frac{L^{(2MN)}(\frac{1}{2},\tilde{\chi}_{d_{0}}\chi'\psi')\chi\psi(d)Q_{d_{0},d_{1}}^{(\chi'\psi')}(\frac{1}{2})}{d{}^{1/2}}V_{\psi,\psi'}^{(\chi,\chi')}\left(\frac{dX}{\gamma_{\psi,\psi'}^{(\chi,\chi')}};t\right),
\end{multline*}
where $\beta_{\psi,\psi}^{(\chi,\chi')}$ and $\gamma_{\psi,\psi'}^{(\chi,\chi')}$
satisfy the bounds listed in Table \ref{tab:coeffs}, and we have
the bounds
\[
\begin{array}{ccc}
V_{\psi,\psi'}^{(\chi,\chi')}\left(\xi;t\right) \ll \vert\xi\vert{}^{-B}(1+\vert t\vert)^{B}
&
\text{and}
&
V(\xi;t) \ll \vert\xi\vert{}^{-B}
\end{array}
\]
uniformly in $\xi$ and $t$ for any number $B>0$.\end{lem}
\begin{proof}
Let $B>0$, $H$ be an even, holomorphic function with $H(0)=1$ satisfying
the growth estimate $H(z)\ll_{\Re z,A}(1+\vert z\vert)^{-A}$ for any $A>0$. We consider the integral
\begin{equation}
I(c,X,t)=\frac{1}{2\pi i}\int_{(1)}X^{cz}\left(\frac{4^{\frac{1}{2}+it+cz}-4}{4^{\frac{1}{2}+it}-4}\right)^{2}Z({\textstyle \frac{1}{2}},{\textstyle \frac{1}{2}}+it+cz;\rho,\rho')H(z)\frac{dz}{z}\label{eq:I-def}
\end{equation}
for a real number $c$, a positive real number $X>0$, and a fixed
real number $t$. Examining the expression when $c=1$, the fraction
cancels the pole of the $Z$ factor at $z=1/2-it$. We apply a shift
of the contour to $\Re z=-1$, picking up the pole at $z=0$, whence
we obtain
\[
I(1,X,t)=Z({\textstyle \frac{1}{2}},{\textstyle \frac{1}{2}}+it;\rho,\rho')+\frac{1}{2\pi i}\int_{(-1)}X^{z}\left(\frac{4^{\frac{1}{2}+it+z}-4}{4^{\frac{1}{2}+it}-4}\right)^{2}Z({\textstyle \frac{1}{2}},{\textstyle \frac{1}{2}}+it+z;\rho,\rho')H(z)\frac{dz}{z}.
\]
We now apply a change of variables $z\mapsto-z$, arriving at 
\[
Z({\textstyle \frac{1}{2}},{\textstyle \frac{1}{2}}+it;\rho,\rho')=I(1,X,t)+I(-1,X,t).
\]
Applying the functional equation (\ref{eq:functional-with-z}) and
the switch of summation formula (\ref{eq:Building-block-sum-switch})
and expanding the $Z$ functions, definition (\ref{eq:Z-definition})
gives that $I(-1,X,t)$ equals
\begin{multline}
X^{it}\sum_{\substack{(\chi,\chi')\in\Phi\\
\psi,\psi'\in\widehat{(\mathbb{Z}/8\mathbb{Z})^{*}}
}
}\beta_{\psi,\psi'}^{(\chi,\chi')}
\sum_{\substack{(m,2MN)=1}
}\frac{L^{(2MN)}(\frac{1}{2},\tilde{\chi}_{m_{0}}\chi'\psi')(\chi\psi)(m)Q_{m_{0},m_{1}}^{(\chi'\psi')}(\frac{1}{2})}{m^{1/2}}V_{\psi,\psi'}^{(\chi,\chi')}\left(\frac{mX}{\gamma_{\psi,\psi'}^{(\chi,\chi')}};t\right),\label{eq:I(-1)}
\end{multline}
where
\[
V_{\psi,\psi'}^{(\chi,\chi')}(\xi;t)=\frac{1}{2\pi i}\int_{(1)}\left(\frac{4^{1/2+it-z}-4}{4^{1/2+it}-4}\right)^{2}\xi^{-z+it}\omega_{\psi,\psi'}^{(\chi,\chi')}(z-it)H(z)\frac{dz}{z}.
\]
We wish to obtain an upper bound for $V_{\psi,\psi'}^{(\chi,\chi')}(\xi;t)$.
Moving the contour to $B$ (recalling that there are no poles of $\omega_{\psi,\psi'}^{(\chi,\chi')}$
in this region), and bounding by taking the absolute value of the
integrand and using the bound (\ref{eq:omega-bound}), we have 
\[
V_{\psi,\psi'}^{(\chi,\chi')}(\xi;t)\ll\vert\xi\vert{}^{-B}(1+\vert t\vert)^{B}
\]
uniformly in $\xi$ and $t$. Changing the summation variable from
$m$ to $d$ in (\ref{eq:I(-1)}), we obtain the first term in the
statement of the lemma.

Looking at $I(1,X,t)$, we have
\[
I(1,X,t)=X^{-it}\sum_{\substack{(d,2MN)=1}
}\frac{L^{(2MN)}({\textstyle \frac{1}{2}},\chi_{d_{0}}\rho)\rho'(d)P_{d_{0},d_{1}}^{(\rho)}({\textstyle \frac{1}{2}})}{d{}^{1/2+it}}V\left(\frac{d}{X};t\right),
\]
where
\[
V(\xi;t)=\frac{1}{2\pi i}\int_{(1)}\left(\frac{4^{\frac{1}{2}+it+z}-4}{4^{1/2+it}-4}\right)^{2}\xi^{-z}H(z)\frac{dz}{z}.
\]
Also, it is immediate that we have the bound $V(\xi;t)\ll\vert\xi\vert{}^{-B}$ uniformly in $\xi$ and $t$.
\end{proof}
We now wish to truncate the sums above, accruing an error. This is
the object of the following lemma.
\begin{lem}
\label{lem:approx-functional-truncated}Let $A$ be a large positive
constant, $t$ be a real number, and $V(\xi;t)$ be a rapidly decaying
function in $\xi$ satisfying
\[
V(\xi;t)\ll\vert\xi\vert{}^{-B}(1+\vert t\vert)^{B},
\]
uniformly in $\xi$ and $t$ for any number $B>0$, let $\chi$ be
a character modulo $k$, and let $a$ be an arithmetic function satisfying $ a(d)\ll d^{\varepsilon} $ uniformly in $d$. Then we can truncate the double sum \textup{
\[
\sum_{\substack{(d,2MN)=1}
}\frac{L^{(2MN)}({\textstyle \frac{1}{2}},\chi_{d_{0}}\chi)\chi'(d)a(d)}{d{}^{1/2}}V\left(\frac{d}{Y};t\right)
\]
at $d<Y^{1+\varepsilon}$, accruing an error that is bounded above
by $ O((1+\vert t\vert)^{2A/\varepsilon}k^{1/4+\varepsilon}Y^{-A}) $.
}\end{lem}
\begin{proof}
The $L$-function is bounded asymptotically by $(d_{0}k)^{1/4+\varepsilon}$
due to the Phragm\'en-Lindelöf convexity bound (\ref{eq:L-bounds}).
The $V$ factor is bounded by its argument to an arbitrarily large
power $-B$. Applying this gives an error that is bounded above by
\[
(1+\vert t\vert)^{2B}\sum_{d>P^{1+\varepsilon}}\frac{(d_{0}k)^{1/4+\varepsilon}}{d^{1/2-\varepsilon}}\left(\frac{d}{Y}\right)^{-B},
\]
and the result follows.
\end{proof}
In order to bound $Z({\textstyle \frac{1}{2}},{\textstyle \frac{1}{2}}+it;\rho,\rho')$,
by applying a smooth partition of unity as in \cite{blomer}, it now
suffices to bound
\[
D_{W,a}(Y;t,\chi'\psi',\chi\psi):=\sum_{\substack{(d,2MN)=1}
}\frac{L^{(2MN)}({\textstyle \frac{1}{2}},\chi_{d_{0}}\chi'\psi')\chi\psi(d)a(d)}{d{}^{1/2}}W\left(\frac{d}{Y};t\right)
\]
for $t$ a real number, $\psi,\psi'\in\widehat{(\mathbb{Z}/8\mathbb{Z})^{*}}$,
a smooth function $W$ with support on $[1,2]$ satisfying
\[
W(x;t)\ll_{B}x{}^{-B}(1+\vert t\vert)^{B}
\]
uniformly in $x$ and $t$ for any $B>0$, an arithmetic function
$a$ satisfying the bound $a(d)\ll d^{\varepsilon}$, and the following
conditions, according to each of the two sums in Lemma \ref{lem:approx-functional}:
either $(\chi,\chi')\in\Phi$ (cf. (\ref{eq:Phi-def})) with conductors
$k$ and $j$ respectively, and
\[
1\leq Y\leq\left(\gamma_{\psi,\psi'}^{(\chi,\chi')} X^{-1}\right)^{1+\varepsilon},
\]
or $(\chi,\chi')=(\rho',\rho)$ and
\[
1\leq Y\leq X^{1+\varepsilon}.
\]
Expanding according to the Dirichlet functional equation, and further
truncating that sum expresses $D_{W,a}(Y;t,\chi',\chi)$ as a double
finite character sum, allowing us to apply Heath-Brown's large sieve
estimate Corollary \ref{cor:heath-brown-large-sieve}. The result
of this is the following lemma.
\begin{lem}
\label{lem:D-bound} We have the bound
\[
D_{W,a}(Y;t,\chi',\chi)\ll(1+\vert t\vert)^{2/\varepsilon}(MN)^{\varepsilon}\left(Y^{1+\varepsilon}+(Yj){}^{1/2+\varepsilon}\right)^{1/2+\varepsilon}
\]
 uniformly in $t$, $Y$, $j$, and $k$.\end{lem}
\begin{proof}
Applying Lemma \ref{lem:approx-functional-truncated} above, we have
\begin{multline*}
D_{W,a}(Y;t,\chi',\chi)=\sum_{\substack{(d,2MN)=1\\
d<Y^{1+\varepsilon}
}
}\frac{L^{(2MN)}(\frac{1}{2},\chi_{d_{0}}\chi)\chi'(d)a(d)}{d{}^{1/2}}W\left(\frac{d}{Y};t\right)+O((1+\vert t\vert)^{2/\varepsilon}j^{1/4+\varepsilon}Y^{-1}).
\end{multline*}
Applying (\ref{eq:L-remove-primes}) and (\ref{eq:L-approx-functional}),
we have that $ D_{W,a}(Y;t,\chi',\chi) $ equals, up to an error of $ O((1+\vert t\vert)^{2/\varepsilon}j^{1/4+\varepsilon}Y^{-1}) $,
\begin{multline}
2\sum_{\substack{(d,2MN)=1\\
d<Y^{1+\varepsilon}
}
}\frac{\chi'(d)}{d^{1/2}}\prod_{p\mid2MN}\left(1-\frac{(\chi_{d_{0}}\chi)(p)}{p^{1/2}}\right)a(d)W\left(\frac{d}{Y};t\right)
\sum_{n=1}^{\infty}\frac{(\chi\chi_{d_{0}})(n)}{n^{1/2}}G_{\kappa_{\chi'}}\left(\frac{n}{\sqrt{c_{0}d_{0}j}}\right),\label{eq:D-1}
\end{multline}
where $c_{0}$ is given in (\ref{eq:Lfunctional-kappa,delta0-1})
and $G_{\kappa}$ is given in (\ref{eq:G-def}). Because of the rapid
decay of $G_{\kappa}$ and $W$, we can truncate the $n$-sum at $n<\left(Yj\right){}^{1/2+\varepsilon}$.

Applying the bounds (\ref{eq:G-bound}) and $a(d)\ll d^{\varepsilon}$,
the error obtained by this is bounded by
\[
(MN)^{\varepsilon}(1+\vert t\vert)^{B}\sum_{d<Y^{1+\varepsilon}}\frac{1}{d^{1/2-\varepsilon}}\left(1+\frac{d}{Y}\right)^{-B}\sum_{n>(Yj)^{1/2+\varepsilon}}\frac{1}{n^{1/2}}\left(1+\frac{n}{\sqrt{c_{0}d_{0}j}}\right)^{-B}
\]
for any large positive number $B$. Indeed, this is bounded above
by
\begin{eqnarray*}
 &  & (MN)^{\varepsilon}(1+\vert t\vert)^{B}Y^{B}j^{B/2}\sum_{d<Y^{1+\varepsilon}}d^{-B-1/2+\varepsilon}\sum_{n>(Yj)^{1/2+\varepsilon}}n^{-B-1/2}d^{B/2}\\
 & \asymp & (MN)^{\varepsilon}(1+\vert t\vert)^{B}Y^{B}j^{B/2}(Y^{1+\varepsilon})^{-B/2+1/2+\varepsilon}(Yj)^{(1/2+\varepsilon)(-B+1/2)}.
\end{eqnarray*}

We see this is bounded above by $(MN)^{\varepsilon}(1+\vert t\vert)^{B}(Yj)^{-\varepsilon B+1}$.
We choose $B$ large enough so that $\varepsilon B-1\geq1$, and so
we can choose $B=2/\varepsilon$. Thus, we have that $ D_{W,a}(Y;t,\chi',\chi) $ equals
\begin{multline*}
2\sum_{\substack{(d,2MN)=1\\
d<Y^{1+\varepsilon}
}
}\prod_{p\mid2MN}\left(1-\frac{(\chi_{d_{0}}\chi)(p)}{p^{1/2}}\right)
\sum_{\substack{n\leq(Yj)^{1/2+\varepsilon}}
}\frac{(\chi\chi_{d_{0}})(n)\chi'(d)}{d^{1/2}n^{1/2}}a(d)W\left(\frac{d}{Y};t\right)G_{\kappa_{\chi'}}\left(\frac{n}{\sqrt{c_{0}d_{0}j}}\right)\\
+O((MN)^{\varepsilon}(1+\vert t\vert)^{2/\varepsilon}(Yj){}^{-1})+O((1+\vert t\vert)^{2/\varepsilon}j^{1/4+\varepsilon}Y^{-1}).
\end{multline*}
We wish to apply Heath-Brown's large sieve Corollary \ref{cor:heath-brown-large-sieve}
to the double sum, but to do this, we need to separate the $n$ and
$d_{0}$ dependence on the $G$ function. Hence, we apply Mellin inversion
to render the main term above as
\begin{multline*}
2\int_{(\varepsilon)}\int_{(\varepsilon)}\tilde{G}_{\kappa_{\chi'}}(s)j^{s/2}\sum_{\substack{d<Y^{1+\varepsilon}}
}\prod_{p\mid2MN}\left(1-\frac{(\chi_{d_{0}}\chi)(p)}{p^{1/2}}\right) \\
\times \sum_{\substack{n<(Yj)^{1/2+\varepsilon}}
}(c_{0}d_{0})^{s/2}a(d)\frac{(\chi\chi_{d_{0}})(n)\chi'(d)}{d^{1/2+w}n^{1/2+s}}\tilde{W}(w;t)Y^{w}dwds.
\end{multline*}
Looking at the summand, for $\Re s=\varepsilon$ and $\Re w=\varepsilon$,
we have
\[
\begin{array}{ccc}
{\displaystyle
(c_{0}d_{0})^{s/2}a(d)\frac{\chi'(d)}{d^{1/2+w}}\ll d^{-1/2+\varepsilon}
}
&
\text{and}
&
{\displaystyle
\frac{\chi(n)}{n^{1/2+s}}\ll n^{-1/2+\varepsilon},
}
\end{array}
\]
so that we can now apply Corollary \ref{cor:heath-brown-large-sieve},
and the result follows.
\end{proof}

\section{Proof of Theorem \ref{thm:subconvexity}}

We first apply Lemma \ref{lem:approx-functional} which gives
\[
Z({\textstyle \frac{1}{2}},{\textstyle \frac{1}{2}}+it;\rho,\rho')=X^{it}\sum_{\substack{(\chi,\chi')\in\Phi\\
\psi,\psi'\in\widehat{(\mathbb{Z}/8\mathbb{Z})^{*}}
}
}\beta_{\psi,\psi'}^{(\chi,\chi')}D_{V,Q}\Biggl(\frac{\gamma_{\psi,\psi'}^{(\chi,\chi')}}{X};t,\chi',\chi\Biggl)+X^{-it}D_{V,P}(X;t,\rho,\rho'),
\]
where the subscripts for $D$ in the first term are $V=V_{\psi,\psi'}^{(\chi,\chi')}(\xi,t)$
and $Q=Q_{d_{0},d_{1}}^{(\chi'\psi')}({\textstyle \frac{1}{2}})$,
and the subscripts for $D$ in the second term are $V=V(\xi,t)$ and
$P=P_{d_{0},d_{1}}^{(\chi\psi)}({\textstyle \frac{1}{2}})$. Applying
Lemma \ref{lem:D-bound} further gives
\begin{multline*}
Z({\textstyle \frac{1}{2}},{\textstyle \frac{1}{2}}+it,\rho,\rho')\ll(1+\vert t\vert)^{2/\varepsilon}(MN)^{\varepsilon}\left[\left(X+(NX)^{1/2}\right)^{1/2+\varepsilon}\right]\\
+(1+\vert t\vert)^{2/\varepsilon}(MN)^{\varepsilon}\max_{\substack{(\chi,\chi')\in\Phi\\
\psi,\psi'\in\widehat{(\mathbb{Z}/8\mathbb{Z})^{*}}
}
}\left[\beta_{\psi,\psi'}^{(\chi,\chi')}\left(\gamma_{\psi,\psi'}^{(\chi,\chi')}X^{-1}+\left(\gamma_{\psi,\psi'}^{(\chi,\chi')}j X^{-1} \right)^{1/2}\right)^{1/2+\varepsilon}\right].
\end{multline*}
Here, $\Phi$ is given by (\ref{eq:Phi-def}), and $j$ is the conductor
of $\chi'$.

We initially set $X=M^{a}N^{b}$ and eventually choose optimized values
for $a$ and $b$, depending on the bounds for the $\beta$'s and
$\gamma$'s from Table \ref{tab:coeffs}. We come to an optimal choice
of $X=M^{2/3}N^{1/3}$ and obtain the subconvexity bound of 
\[
M^{1/3+\varepsilon}N^{1/6+\varepsilon}+M^{1/6+\varepsilon}N^{1/3+\varepsilon}\asymp(MN(M+N)){}^{1/6+\varepsilon}
\]
 (even in the case of $M=N$), compared to the convexity bound of
$(MN)^{3/8+\varepsilon}$.

\section{\label{sec:Non-vanishing}Proof of Theorem \ref{thm:nonvanishing}}

Here we present an application of the theory of double Dirichlet series as developed. Given a fixed positive prime $N$, we seek an upper bound on $d$ such that $L(1/2,\chi_{dN})$ does not vanish. We follow a modified version of a method outlined in \cite{hk}, and we rigorously prove an important lower bound on coefficients which arise from a residue of $Z$. Let $h(y)$ be a smooth weight function as in Definition \ref{def:smooth-weight-function}.
Expanding as per (\ref{eq:Z-definition}) and by Mellin inversion
we have
\begin{equation}
\label{eq:int-sum}
\int_{(2)}\tilde{h}(w)Z({\textstyle \frac{1}{2}},w;\chi_{N},\psi_{1})X^{w}dw
= \sum_{\substack{(d,2N)=1}
}L^{(2N)}({\textstyle \frac{1}{2}},\chi_{d_{0}N})P_{d_{0},d_{1}}^{(\chi_{N})}({\textstyle \frac{1}{2}})h({\textstyle \frac{d}{X}}).
\end{equation}

We move the contour of the integral on the left-hand side to $\Re w=-\varepsilon$,
picking up a residue at $w=1$ due to the double pole of $Z({\textstyle \frac{1}{2}},w;\chi_{N},\psi_{1})$
there. If we write its Laurent expansion as
\[
Z(1/2,w;\chi_{N},\psi_{1})=\frac{\mu_{N}}{(w-1)^{2}}+\frac{\nu_{N}}{(w-1)}+\cdots
\]
then the left-hand side of (\ref{eq:int-sum}) equals
\[
[\nu_{N}\tilde{h}(1)+\mu_{N}\tilde{h}'(1)]X+\mu_{N}\tilde{h}(1)X\log X
+\int_{(-\varepsilon)}\tilde{h}(w)Z({\textstyle \frac{1}{2}},w;\chi_{N},\psi_{1})X^{w}dw.
\]
We now apply the symmetric functional equation (\ref{eq:functional-3}) with $\rho\psi=\chi_N$, $\rho'=\psi'=\psi_1$, and $(s,w)$=$(1/2,w)$ with $\Re w=-\varepsilon$ to $Z$ in the resulting integral. Using Lemma \ref{lem:functional-coefficients} to bound the coefficients, we see that $C_{\rho\tilde{\chi}_n}=N$ when $n=1$ and is unity otherwise, and $C_{\rho'\tilde{\chi}_m}^{1/2-s}=1$ in every case because $s=1/2$. In every case the $A$ factors will always be bounded above by $O(N^\varepsilon)$, since $\Re w = -\varepsilon$. Also, $ A_m^{(\rho \tilde{\chi}_n)}(s+w-1/2) $ vanishes when $m=N$ and $n=1$. Hence, we have
\begin{multline*}
	Z(1/2,w;\chi_N,\psi_1)
	=\sum_{\psi^{\star}, \psi^{\star\star}, \psi^{\star\star\star}\in\widehat{(\mathbb{Z}/8\mathbb{Z})^{*}}}
	\sum_{m,r\mid N} 
	B(w;N,m,r,\psi^\star,\psi^{\star\star},\psi^{\star\star\star})
	\\
	\times
	\left(	
	N^{1/2-w} Z(1-w,1/2;\chi_N \tilde{\chi}_r \psi^\star \psi^{\star\star\star}, \psi^{\star\star})
	+ Z(1-w,1/2;\tilde{\chi}_r \psi^\star \psi^{\star\star\star}, \tilde{\chi}_m\psi^{\star\star})
	\right),
\end{multline*}
where, due also to the bound (\ref{eq:functional-3-coeffs-bound}), we have
\[
	\vert
		B(w;N,m,r,\psi^\star,\psi^{\star\star},\psi^{\star\star\star})
	\vert
	\ll
	N^\varepsilon (1+\vert w \vert)^{1+\varepsilon}.
\]
We now apply the bound (\ref{eq:convex-trivial-2}), which finally gives
\[
\vert Z({\textstyle \frac{1}{2}},-\varepsilon+it;\chi_{N},\psi_{1}) \vert
\ll
(1+\vert t\vert)^{1+\varepsilon} N^{1/2+\varepsilon},
\]
which we shall apply to the integral. We thus have
\begin{equation}
S(X;\chi_{N}):=\sum_{\substack{(d,2N)=1}
}L^{(2N)}({\textstyle \frac{1}{2}},\chi_{d_{0}N})P_{d_{0},d_{1}}^{(\chi_{N})}({\textstyle \frac{1}{2}})h({\textstyle \frac{d}{X}})=a_{N}X\log X+b_{N}X+O(N^{1/2+\varepsilon})\label{eq:S-asymptotic-arbitrary-main-term}
\end{equation}
for certain coefficients $a_{N}$ and $b_{N}$. More elementary analysis
of $S(X;\chi_{N})$ via Theorem \ref{thm:L-moment-asymptotic} (cf.
\S\ref{sec:L-function-sum-asymptotic}) below gives us the lower bounds
$a_{N},b_{N}\gg N^{-\varepsilon}$. If we now assume that 
\[
L^{(2N)}({\textstyle \frac{1}{2}},\chi_{dN})=L({\textstyle \frac{1}{2}},\chi_{d_{0}N})\prod_{p\mid2d_{1}N}(1-p^{-1/2})
\]
 vanishes for $d\ll X$, then we get a contradiction as long as $a_{N}X\log X$
is greater than the error term. We see we can choose $X=N^{1/2+\varepsilon}$,
as required.

We essentially combine two asymptotic formulas for $S(X;\chi_{N})$:
Looking at it elementarily via Theorem \ref{thm:L-moment-asymptotic}
from \S\ref{sec:L-function-sum-asymptotic} gives us a bad error term,
but lower bounds on the coefficients. Looking at it analytically as
above allows us to take advantage of the bound (\ref{eq:convex-trivial-2})
for $Z$ in order to obtain a smaller
error term.

\section{\label{sec:L-function-sum-asymptotic}An Asymptotic Formula for an $L$-function sum}

\subsection{Result}

Of particular importance for the smallest nonvanishing quadratic central
value of twisted $L$-functions result is an asymptotic formula for
the weighted and twisted $L$-function sum given by $S(X;\chi)$ in
(\ref{eq:S-asymptotic-arbitrary-main-term}). Indeed, we already have
an asymptotic formula, but it is important that we have a lower bound
on the main term coefficients.

Let $N$ be a natural number, $X$ a large positive real number, and
let $\chi$ be a quadratic primitive character modulo $N$. Let $h$
be a smooth weight function as defined in Definition \ref{def:smooth-weight-function},
and define
\[
S(X;\chi)=\sum_{\substack{(d,2N)=1}
}L^{(2N)}({\textstyle \frac{1}{2}},\chi_{d_{0}}\chi)P_{d_{0},d_{1}}^{(\chi)}({\textstyle \frac{1}{2}})h(d/X).
\]
We seek to obtain information on the main term coefficients of the
asymptotic formula (\ref{eq:S-asymptotic-arbitrary-main-term}). We
shall prove the following Theorem.
\begin{thm}
\label{thm:L-moment-asymptotic}There exists $\delta>0$ such that
we have the asymptotic formula
\[
S(X;\chi)=a_{N}X\log X+b_{N}X+O(N^{3/8+\varepsilon}X^{1-\delta}),
\]
where 
\[
N^{-\varepsilon}\ll a_{N},b_{N}\ll N^{\varepsilon},
\]
uniformly in $N$.
\end{thm}

\subsection{Proof of Theorem \ref{thm:L-moment-asymptotic}}

According to the approximate functional equation for Dirichlet $L$-functions
(\ref{eq:L-approx-functional}), we have
\[
L({\textstyle \frac{1}{2}},\chi_{d_{0}}\chi)=2\sum_{n=1}^{\infty}\frac{(\chi_{d_{0}}\chi)(n)}{n^{1/2}}G_{\kappa}\left(\frac{n}{\sqrt{c_{0}d_{0}N}}\right),
\]
where $G_{\kappa}$ is given in (\ref{eq:G-def}), and $c_{0}$ is
given in (\ref{eq:Lfunctional-kappa,delta0-1}). From the definition
of $P_{d_{0},d_{1}}^{(\chi)}$ in Theorem \ref{thm:correction-factors}
we have
\[
P_{d_{0},d_{1}}^{(\chi)}({\textstyle \frac{1}{2}})=\sum_{f\mid d_{1}^{2}}\frac{\mu(f_{0})(\chi_{d_{0}}\chi)(f_{0})}{f_{0}^{1/2}},
\]
where we write $f=f_{0}f_{1}^{2}$ with $f_{0}$ squarefree. We shall
also use the expansion
\[
1-\frac{(\chi_{d_{0}}\chi)(2)}{2^{1/2}}=\sum_{g\mid2}\frac{\mu(g)(\chi_{d_{0}}\chi)(g)}{g^{1/2}}.
\]
Applying these expressions to $S(X;\chi)$, we have
\begin{multline}
S(X;\chi)=2\sum_{\substack{(d_{1},2N)=1}
}\sum_{\substack{(d_{0},2N)=1}
}\sum_{f_{1}\mid d_{1}}\sum_{g\mid2f_{1}^{2}}\frac{\mu(g)(\chi_{d_{0}}\chi)(g)}{h^{1/2}}\sum_{n=1}^{\infty}\frac{(\chi_{d_{0}}\chi)(n)}{n^{1/2}}G_{\kappa}\left(\frac{n}{\sqrt{c_{0}d_{0}N}}\right)h\left(\frac{d_{0}d_{1}^{2}}{X}\right).\label{eq:S-average-L}
\end{multline}

For a subset $H\subset\mathbb{N}$, we define $S_{H}(X;\chi)$ to
be the same as the expression (\ref{eq:S-average-L}) with the added
condition in the $n$-sum of $ng\in H$. We use Mellin inversion for
$G_{\kappa}$ to separate the variables, and by moving the $d_{0}$-sum
to the inside, we get that $S_{H}(X;\chi)$ equals
\begin{multline}
2\int_{(\frac{1}{2}+\varepsilon)}N^{\frac{s}{2}}\sum_{\substack{(d_{1},2N)=1}
}\sum_{f_{1}\mid d_{1}}\sum_{\substack{g\mid2f_{1}^{2}\\
(g,N)=1
}
}\frac{\mu(g)\chi(g)}{g^{1/2}}\sum_{\substack{ng\in H\\
(n,N)=1
}
}\frac{\chi(n)}{n^{s+\frac{1}{2}}}
\sum_{\substack{(d_{0},2N)=1}
}\tilde{G}{}_{\kappa}(s)\chi_{d_{0}}(ng)c_{0}^{s/2}d_{0}^{s/2}h\left(\frac{d_{0}d_{1}^{2}}{X}\right)ds.\label{eq:S_H-expanded}
\end{multline}
The variables $\kappa$ and $c_{0}$ depend on the residues of $d_{0}$
and $N$ modulo 4. Thus, given $\ell\in\lbrace\pm1\rbrace$, we define
$\kappa(\ell)$ and $c_{0}(\ell)$ to be the corresponding values
for $d_{0}\equiv\ell\ (\text{mod }4)$. For convenience, for $\iota\in\lbrace\pm1\rbrace$,
we define $S_{H}(X;\chi,\iota,\ell)$ to be the same as (\ref{eq:S_H-expanded})
except that $\kappa$ and $c_{0}$ are replaced with $\kappa(\ell)$
and $c_{0}(\ell)$, and $\psi_{\iota}(d_{0})$ is multiplied to the
summand in the $d_{0}$-sum.

Observing that $\frac{1}{2}(1\pm\psi_{-1}(d_{0}))$ is the characteristic
function of $d_{0}\equiv\pm1\ (\text{mod }4)$, the $d_{0}$-sum in
the integrand of (\ref{eq:S_H-expanded}) is
\[
\frac{1}{2}\sum_{\substack{(d_{0},2N)=1}
}\sum_{\pm}(1\pm\psi_{-1}(d_{0}))\tilde{G}{}_{\kappa(\pm1)}(s)c_{0}(\pm1)^{s/2}\tilde{\chi}_{ng}(d_{0})d_{0}{}^{s/2}h\left(\frac{d_{0}d_{1}^{2}}{X}\right),
\]
whence we see that
\begin{equation}
S_{H}(X;\chi)=\frac{1}{2}\sum_{\iota=\pm1}\sum_{\ell=\pm1}\text{sgn}(1+\iota+\ell)S_{H}(X;\chi,\iota,\ell).\label{eq:S_H-sieve-d_0-mod-4}
\end{equation}
For treatment of the $d_{0}$-sum, for a positive real number $Y$
and a Dirichlet character $\psi$ we further define 
\[
T(s;Y,\psi)=\sum_{d_{0}=1}^{\infty}\psi(d_{0})d_{0}^{s/2}h\left(\frac{d_{0}}{Y}\right),
\]
where the sum is over squarefree $d_{0}$. With these simplifications,
we obtain that $S_{H}(X,\chi,\iota,\ell)$ equals
\begin{multline}
2\int_{(1/2+\varepsilon)}(Nc_{0}(\ell))^{s/2}\tilde{G}{}_{\kappa(\ell)}(s)\sum_{\substack{(d_{1},2N)=1}
}\sum_{f_{1}\mid d_{1}}\sum_{\substack{g\mid2f_{1}^{2}\\
(g,N)=1
}
}\frac{\mu(g)\chi(g)}{g^{1/2}}
\sum_{\substack{ng\in H\\
(n,N)=1
}
}\frac{\chi(n)}{n^{s+1/2}}T(s;X/d_{1}^{2},\chi_{0}^{(2N)}\tilde{\chi}_{ng}\psi_{\iota})ds.\label{eq:S_H-expanded-2}
\end{multline}
We shall choose $H=\square$ which we use to denote the set of positive
squares, and denoting the complement of $H$ in $\mathbb{N}$ by $\bar{H}$,
it is clear that
\[
S(X;\chi)=S_{\square}(X;\chi)+S_{\bar{\square}}(X;\chi).
\]

In order to estimate the size of $T(s;Y,\psi)$, we observe that there
are order $Y$ squarefree numbers up to $Y$. If $\psi$ is a principal
character and $\Re s=\varepsilon$, we therefore expect this sum to
be roughly of size $Y^{1+\varepsilon}$. If $\psi$ is non-principal,
the oscillations will give us a P\'olya-Vinogradov type estimate.

Indeed, with this last point in mind, we step back to explain the
main idea of the proof: The sum $S_{H}(X,\chi,\psi_{\iota},\ell)$
will hence only be large when $\tilde{\chi}_{ng}\psi_{\iota}$ is
principal, that is, precisely when $H=\square$ and $\iota=1$, and
will be small otherwise. To this end, we have the following asymptotic
formula.
\begin{lem}
\label{lem:T-asymptotics}For $\Re s>0$, we have
\[
T(s;Y,\chi_{0}^{(m)})=\frac{\tilde{h}(1+s/2)}{\zeta(2)}\prod_{p\mid m}\left(1+\frac{1}{p}\right)^{-1}Y{}^{1+s/2}+U(s;Y,m)
\]
with $U(s;Y,m)$ holomorphic and
\[
U(s;Y,m)\ll\vert s\vert^{1/2-\Re s/2+\varepsilon}m^{\varepsilon}Y^{1/2+\Re s/2+\varepsilon},
\]
uniformly in $Y$ and $\Im s$. Further, if $\psi=\chi_{0}^{(m)}\tilde{\psi}$,
where $\tilde{\psi}$ is a nontrivial quadratic primitive character
with conductor $c$, then
\[
T(s;Y,\psi)\ll\vert cs\vert^{1/2-\Re s/2+\varepsilon}m{}^{\varepsilon}Y^{1/2+\Re s/2+\varepsilon}
\]
uniformly in $\Im s$, m, and $Y$.\end{lem}
\begin{proof}
Via Mellin inversion, we have
\[
T(s;Y,\chi_{0}^{(m)}\tilde{\psi})=\int_{(1+\Re s/2+\varepsilon)}\mathcal{T}_{s}(z)\tilde{h}(z)Y^{z}dz,
\]
where we have the generating function
\begin{eqnarray*}
\mathcal{T}_{s}(z) & = & \sum_{d=1}^{\infty}\frac{\mu^{2}(d)(\chi_{0}^{(m)}\tilde{\psi})(d)}{d^{z-s/2}}=\frac{L(z-s/2,\tilde{\psi})}{L(2z-s,\tilde{\psi})}\prod_{p\mid m}\left(1+\frac{1}{p^{z-s/2}}\right)^{-1}.
\end{eqnarray*}
If $\tilde{\psi}$ is the trivial character then the $L$-function
in the numerator is just the zeta function, and therefore has a pole
at $z=1+s/2$. Due to the $1/L(2z-s,\tilde{\psi})$ factor, all other
poles lie in $\Re z<\Re s/2$. We move the contour to $(1/2+\Re s/2+\varepsilon)$,
and in the case where $\tilde{\psi}$ is trivial, we pick up the residue
\begin{equation}
\frac{\tilde{h}(1+s/2)}{\zeta(2)}\prod_{p\mid m}\left(1+\frac{1}{p}\right)^{-1}Y{}^{1+s/2}.\label{eq:T-main-term}
\end{equation}
Due to (\ref{eq:w-bound}) along with the convexity bound
\[
L(z-s/2,\tilde{\psi})\ll[c(1+\vert\Im z-\Im s/2\vert)]^{1/2-\Re s/2+\varepsilon}
\]
obtained from (\ref{eq:L-bounds}), we bound the resulting integral
by 
\[
\vert cs\vert^{1/2-\Re s/2+\varepsilon}m^{\varepsilon}Y{}^{1/2+\Re s/2+\varepsilon}.
\]

\end{proof}

\subsubsection{Main term}

As explained above, the main contribution will come from $S_{\square}(X;\chi,1,\ell)$,
which we will bound from below. Looking at the expansion (\ref{eq:S_H-expanded-2})
with $\iota=\ell=1$ and $H=\square$, since $ng$ is square, and
recalling that $g$ is squarefree, we have $n=gm{}^{2}$ for $m\in\mathbb{N}$,
so the inner sum is 
\[
\sum_{\substack{ng\in\square\\
(n,N)=1
}
}\frac{1}{n^{s+1/2}}T(s;X/d_{1}^{2},\chi_{0}^{(2ngN)})=\frac{1}{g^{s+1/2}}\sum_{\substack{(m,N)=1}
}\frac{1}{m^{2s+1}}T(s;X/d_{1}^{2},\chi_{0}^{(2gmN)}).
\]
By Lemma \ref{lem:T-asymptotics}, the above expression is
\begin{multline*}
\frac{1}{g^{s+1/2}}\sum_{\substack{(m,N)=1}
}\frac{1}{m^{2s+1}}\left[\prod_{\substack{p\mid2gmN}
}\left(1+\frac{1}{p}\right)^{-1}\frac{\tilde{h}(1+s/2)}{\zeta(2)}\left(\frac{X}{d_{1}^{2}}\right)^{1+s/2}+U(s;X/d_{1}^{2},2gmN)\right]
\end{multline*}
where $U(s;X/d_{1}^{2},2gmN)$ is holomorphic for $\Re s>0$ and 
\[
U(s;X/d_{1}^{2},2gmN)\ll\vert s\vert^{1/2-\Re s/2+\varepsilon}(gmN)^{\varepsilon}\left(X/d_{1}^{2}\right)^{1/2+\Re s/2}.
\]
Referring to (\ref{eq:S_H-expanded-2}), and moving the contour to
$(\varepsilon)$, the error term is bounded above by
\begin{multline}
\sum_{\substack{(d_{1},2N)=1}
}\int_{(\varepsilon)}\vert Nc_{0}(\ell)\vert^{\Re s/2}\vert\tilde{G}{}_{\kappa(\ell)}(s)\vert\vert s\vert^{1/2-\Re s/2+\varepsilon}
\sum_{f_{1}\mid d_{1}}\sum_{\substack{g\mid2f_{1}^{2}\\
(g,N)=1
}
}\frac{1}{g^{s+1}}\sum_{\substack{(m,N)=1}
}\frac{(gmN)^{\varepsilon}}{m^{2s+1}}\left(\frac{X}{d_{1}^{2}}\right)^{1/2+\Re s/2}ds.\label{eq:S-square-error}
\end{multline}
Bounding the sums absolutely and using the fact that $\tilde{G}{}_{\kappa(\ell)}(s)$
decays rapidly in fixed vertical strips, we see that this is bounded
above by $N^{\varepsilon}X^{1/2+\varepsilon}$. As for the main term,
through a calculation we have the following result.
\begin{lem}
\label{lem:Euler-E-composition}We have the identity
\[
\sum_{\substack{(m,N)=1}
}\frac{1}{m^{2s+1}}\prod_{\substack{p\mid2gmN}
}\left(1+\frac{1}{p}\right)^{-1}=\zeta(2s+1)E_{0}(s)E_{1}(s;g),
\]
where
\[
E_{0}(s)=\frac{4}{9}(1-2^{-2s-1})\prod_{p}(1+p^{-1})(1+p^{-1}(1-p^{-2s-1})^{-1})^{-1}\prod_{p\mid N}(1+p^{-1})^{-2}(1+p^{-1}-p^{-2s-1}),
\]
\[
E_{1}(s;g)=\prod_{\substack{p\mid g\\
g\neq2
}
}(1+p^{-1})^{-2}(1-p^{-2s-1})^{-1}(1+p^{-1}-p^{-2s-1}).
\]
\end{lem}
\begin{proof}
This is a straightforward but monotonous calculation which we omit.
\end{proof}
Applying this, we now have that $S_{\square}(X,\chi,1,\ell)$ equals, up to an error $O(N^{\varepsilon}X^{1/2+\varepsilon})$,
\begin{multline*}
\frac{2}{\zeta(2)}\int_{(\varepsilon)}\tilde{h}(1+{\textstyle \frac{s}{2}})(c_{0}(\ell)N)^{s/2}X^{1+s/2}\tilde{G}{}_{\kappa(\ell)}(s)\zeta(2s+1)E_{0}(s)
\sum_{\substack{(d_{1},2N)=1}
}d_{1}^{-2-s}\sum_{f_{1}\mid d_{1}}\sum_{\substack{g\mid2f_{1}^{2}\\
(g,N)=1
}
}\frac{\mu(g)}{g^{s+1}}E_{1}(s;g)ds.
\end{multline*}
Since $E(s;g)$ is multiplicative in $g$, we can further collapse
the $g$-sum above into an Euler product. Hence we have
\begin{equation}
S_{\square}(X,\chi,1,\ell)=\frac{2}{\zeta(2)}\int_{(\varepsilon)}\tilde{h}(1+s/2)(c_{0}(\ell)N)^{s/2}X^{1+s/2}\tilde{G}{}_{\kappa(\ell)}(s)\zeta(2s+1)E_{0}(s)H(s)ds+O(N^{\varepsilon}X^{1/2+\varepsilon}),\label{eq:S-1-ell}
\end{equation}
where
\[
H(s) = \sum_{\substack{(d_{1},2N)=1}
}d_{1}^{-2-s}\sum_{f_{1}\mid d_{1}}\prod_{p\mid2f_{1}}\left(1-p^{-s-1}E_{1}(s;p)\right).
\]
We have the following estimates for $H$.
\begin{lem}
\label{lem:H-bounds}There exists $K>0$ such that $1/3\leq H(0)\leq K$
and $\vert H'(0)\vert\leq K$.\end{lem}
\begin{proof}
Let $\Re s\geq0$, and for convenience define
\[
H(s;d_{1})=\sum_{f_{1}\mid d_{1}}\prod_{p\mid2f_{1}}\left(1-p^{-s-1}E_{1}(s;p)\right).
\]
Because $0<E_{1}(0;p)\leq 4/3$, we have $1/3\leq1-p^{-1}E_{1}(0;p)<1$,
and so taking $d_{1}=1$, we have $H(0)\geq1/3$. Hence by the same
reasoning,
\[
\vert H(0;d_{1})\vert\leq\sum_{n\mid d_{1}}1\ll d_{1}^{\varepsilon},
\]
and so we see that $H(0)$ is absolutely bounded above. In order to
show that $H'(0)$ is absolutely bounded, because we have
\[
H'(0)=\sum_{\substack{(d_{1},2N)=1}
}d_{1}^{-2}(H'(0;d_{1})-H(0;d_{1})\log d_{1}),
\]
it suffices to show that $H'(0;d_{1})\ll d_{1}^{\varepsilon}$, and
for this, it suffices to show that $E_{1}'(0;p)$ is absolutely bounded,
which is easily seen via taking the logarithmic derivative.
\end{proof}
We shall also need the following bounds for the $E_{0}$ function.
\begin{lem}
\label{lem:E-bounds}We have the bounds
\[
N^{-\varepsilon}\ll E_{0}(0),\ E_{0}'(0)\ll N^{\varepsilon}.
\]
\end{lem}
\begin{proof}
We have
\[
N^{-\varepsilon} \ll d(N) \ll \frac{2}{9\zeta(2)}\prod_{p\mid N}(1+p^{-1})^{-2}
=E_{0}(0)\ll 1,
\]
and the result follows.

To treat the derivative, it is easily shown that the logarithmic derivative at $s=0$ is bounded below by a constant and above by $N^\varepsilon$, as required.
\end{proof}
We wish to move the contour of integration of (\ref{eq:S-1-ell})
from $(\varepsilon)$ to $(-\varepsilon)$. In doing so, we pick up
a double pole, since $\tilde{G}_{\kappa(\ell)}(s)$ and $\zeta(2s+1)$
have simple poles at $s=0$. The residue of this pole shall be our
main term. In order to calculate it, we shall need further analysis
of the integrand. We define
\[
A(s)=\tilde{h}(1+s/2)N^{s/2}X^{1+s/2}E_{0}(s)H(s)
\]
which is the part of the integrand which is holomorphic at $s=0$.
If the Laurent coefficients of $\zeta(2s+1)$ and $\tilde{G}_{\kappa(\ell)}$
(centred at $0$) are given by $e_{n}$ and $g_{n}$ respectively,
then the residue is
\begin{equation}
R:=(g_{-1}e_{0}+g_{0}e_{-1})A(0)+g_{-1}e_{-1}A'(0),\label{eq:int-residue-coeff}
\end{equation}
and
\begin{equation}
A(0)=\tilde{h}(1)E_{0}(0)H(0)X,\label{eq:A-expanded}
\end{equation}
\begin{multline}
A'(0)=\frac{1}{2}\tilde{h}'(1)E_{0}(0)H(0)X+\frac{1}{2}\tilde{h}(1)E_{0}(0)H(0)(\log N)X\\
+\frac{1}{2}\tilde{h}(1)E_{0}(0)H(0)X\log X+\tilde{h}(1)E_{0}'(0)H(0)X+\tilde{h}(1)E_{0}(0)H'(0)X.\label{eq:A'-expanded}
\end{multline}
We now prove the following expression of the residue $R$ from the
above results.
\begin{cor}
For sufficiently large $N$, we have
\[
R=a_{N}X\log X+b_{N}X,
\]
where
\[
N^{-\varepsilon}\ll a_{N},b_{N}\ll N^{\varepsilon}.
\]
\end{cor}
\begin{proof}
The first term comes from the $X\log X$ term in (\ref{eq:A'-expanded}).
We see that $g_{-1}=1$ due to the definition of $\tilde{G}_{\kappa}(s)$
given in (\ref{eq:G-def}), and we also easily see that $e_{-1}=1/2$,
so that the coefficient for $A'(0)$ in (\ref{eq:int-residue-coeff})
is positive. Now the bounds for $a_{N}$ follow from those for $E_{0}(0)$
and $H(0)$ from Lemmas \ref{lem:H-bounds} and \ref{lem:E-bounds},
and the fact that $\tilde{h}(1)$ is simply a positive constant.

Next, we prove the bounds for the $X$ term coefficient. The upper
bound follows from those for $E_{0}(0)$ and $H(0)$ in Lemmas \ref{lem:H-bounds}
and \ref{lem:E-bounds}, and again from the fact that $\tilde{h}(1)$
and $\tilde{h}'(1)$ are constants. As for the lower bound, there
are two difficulties: First, we do not know if the coefficient for
$A(0)$ in (\ref{eq:int-residue-coeff}) is negative, which would
result in a negative $X$ term contribution from (\ref{eq:A-expanded}).
Secondly, we do not know whether the $\tilde{h}'(1)$ factor in the
first term of (\ref{eq:A'-expanded}) is positive. Nonetheless, we
can simply compare the (positive) $X$ coefficient in the second term
of (\ref{eq:A'-expanded}) to that of the two terms just mentioned,
and choose $N$ large enough so that the $\log N$ factor dominates.
The lower bound then follows again from Lemmas \ref{lem:H-bounds}
and \ref{lem:E-bounds}.
\end{proof}
Now we just need to bound the remaining integral (\ref{eq:S-1-ell})
with contour $(-\varepsilon)$. Using the triangle inequality and
observing that $\tilde{G}{}_{\kappa(\ell)}$ decays rapidly in fixed
vertical strips, we therefore see that the integral with contour $(-\varepsilon)$
is bounded above by $N^{\varepsilon}X^{1-\varepsilon/2}.$

\subsubsection{Error term}

Using the same method for bounding the error term in the previous
section, we can bound $S_{\square}(X;\chi,-1,\ell)$ from above, since
its expansion according to (\ref{eq:S_H-expanded-2}) will have the
factor $T(s;X/d_{1}^{2},\chi_{0}^{(2ngN)}\psi_{-1})$. By Lemma \ref{lem:T-asymptotics}
this becomes (\ref{eq:S-square-error}). It now remains to bound $S_{\bar{\square}}(X;\chi,\iota,\ell)$
from above.

According to (\ref{eq:S_H-expanded}) with the integral contour moved
to $(3/4+\varepsilon)$, we see that $S_{\bar{\square}}(X;\chi,\iota,\ell)$ equals
\begin{multline*}
2\int_{(3/4+\varepsilon)}(c_{0}(\ell)N)^{s/2}\tilde{G}{}_{\kappa(\ell)}(s)\sum_{\substack{(d_{1},2N)=1}
}\sum_{f_{1}\mid d_{1}}\sum_{\substack{g\mid2f_{1}^{2}\\
(g,N)=1
}
}\frac{\mu(g)\chi(g)}{g^{1/2}}
\sum_{\substack{ng\in\bar{\square}\\
(n,N)=1
}
}\frac{\chi(n)}{n^{s+1/2}}T(s;X/d_{1}^{2},\chi_{0}^{(2N)}\tilde{\chi}_{ng}\psi_{\iota})ds.
\end{multline*}
Applying Lemma \ref{lem:T-asymptotics}, since $ng$ is not square,
the character $\chi_{0}^{(2N)}\tilde{\chi}_{ng}\psi$ is never principal,
so we have
\[
T(s;X/d_{1}^{2},\chi_{0}^{(2N)}\tilde{\chi}_{ng}\psi_{\iota})\ll(\vert s\vert ng)^{1/4+\varepsilon}N^{\varepsilon}(X/d_{1}^{2}){}^{7/8+\varepsilon}.
\]
Now we can absolutely bound the sums and ignore the condition that
$ng$ is not square, as we did for (\ref{eq:S-square-error}). We
note that the $n$-sum will absolutely converge since $\Re s=3/4+\varepsilon$,
as long as we select a small enough $\varepsilon$ in the bound for
the $T$-function above. We then arrive at a bound of
\[
S_{\bar{\square}}(X;\chi,\iota,\ell)\ll N^{3/8+\varepsilon}X^{7/8+\varepsilon},
\]
whence by (\ref{eq:S_H-sieve-d_0-mod-4}) we have sufficiently bounded
$S_{\bar{\square}}(X;\chi)$.

\section{\label{sec:Other-methods}Appendix: Subconvexity bounds applied to non-vanishing results}

Though the non-vanishing result Theorem \ref{thm:nonvanishing} proven in \S \ref{sec:Non-vanishing} does not require the subconvexity bound Theorem \ref{thm:subconvexity}, one might ask whether a subconvexity bound could be applied to such a problem. We positively answer this here, outlining another method presented in \cite{hk}.

In \S \ref{sec:Non-vanishing}, we start with (\ref{eq:int-sum}) and move the contour of the integral from $(2)$ to $(-\varepsilon)$, picking up a residue at $w=1$. Instead, we can move the contour to $(1/2)$. Now it is clear that a good subconvexity bound for $Z({\textstyle 1/2},w;\chi_N,\psi_1)$ with $\Re w = 1/2$ will produce a non-vanishing result. In our case, the subconvexity bound Theorem \ref{thm:subconvexity} will only yield an upper bound of $N^{2/3+\varepsilon}$, which is worse than that already presented. However, this formulation gives motivation for developing subconvexity bounds for double Dirichlet series. In particular, an improvement to Theorem \ref{thm:nonvanishing} would result from an improvement in the $N$ exponent lower than $1/4+\varepsilon$ in Theorem \ref{thm:subconvexity}.

\section*{Funding}

This work was supported by the Ontario Graduate Scholarship Program; and the European Research Council Starting Grant [258713].

\section*{Acknowledgements}

I sincerely thank Valentin Blomer for suggesting the problem and many useful ideas. I also sincerely thank the following individuals who provided useful suggestions and corrections: Solomon Friedberg, John Friedlander, Florian Herzig, Henry Kim, and Stephen Kudla.

\bibliographystyle{amsalpha}
\bibliography{paper-arxiv2}

\end{document}